\documentclass{article}
\usepackage{graphicx}%
\usepackage{multirow}%
\usepackage{amsmath,amssymb,amsfonts}%
\usepackage{amsthm}%
\usepackage{mathrsfs}%
\usepackage[title]{appendix}%
\usepackage{xcolor}%
\usepackage{textcomp}%
\usepackage{manyfoot}%
\usepackage{booktabs}%
\usepackage{algorithm}%
\usepackage{algorithmicx}%
\usepackage{algpseudocode}%
\usepackage{listings}%
\usepackage{mathtools}
\usepackage{comment}
\usepackage{hyperref}
\usepackage{caption}
\usepackage[hmargin={26mm, 26mm}, vmargin={15mm, 25mm}]{geometry}
\newtheorem{theorem}{Theorem}
\newtheorem{proposition}[theorem]{Proposition}%
\newtheorem{lemma}[theorem]{Lemma}
\newtheorem{corollary}[theorem]{Corollary}

\newtheorem{remark}{Remark}%

\newtheorem{definition}{Definition}%
\usepackage{empheq}
\newenvironment{simultempheq}[2]{
        \subequations \label{#2}
        
        \empheq[left = \empheqlbrace \, ]{#1}
}{
        \endempheq
        \endsubequations
}
\let\Re\relax
\DeclareMathOperator{\Re}{Re}

\DeclareMathOperator{\cl}{cl}
\begin{document}
\title{\rightline{\normalsize{Preprint.}}
\vskip2\baselineskip
Dynamical Systems in Elliptical Pursuit and Evasion
\author{Sota Yoshihara$^{1}$ \\
\small \ $^{1}$Graduate School of Mathematics, Nagoya University, Japan\\ 
\small email:~sota.yoshihara.e6@math.nagoya-u.ac.jp\\
\small WWW homepage:~https://researchmap.jp/yoshiharasota}
}
\date{}
\maketitle

\abstract{
This paper investigates the elliptical case of one-on-one pursuit and evasion problems. Using the simultaneous differential equations derived by Barton and Eliezer, we derive a dynamical system based on the assumption that the shape of the pursuer's trajectory is unaffected by the evader's speed. The dynamical system involves the angular difference between the velocity vectors of the players and their separation distance. We examine whether the pursuer lies inside or outside the evader's elliptical path, and how the angle between the directions of motion of the two players evolves. Moreover, using a complex variable that includes information about the logarithmic distance and the angular difference, we transform the system into a single ordinary differential equation along which the distance between any two solutions never increases. We obtain four main results. First, when the evader is slower than the pursuer, the pursuer captures the evader in finite time, and we derive an explicit upper bound for the capture time. Second, when the evader is faster than the pursuer, a pursuer starting inside the evader's elliptical path is never captured, whereas capture does occur from just outside it on a nearly head-on course. Third, in the same case, the system has a unique periodic solution, to which every solution that avoids capture converges. Fourth, this periodic solution is hyperbolic and exponentially stable, and the product of its Floquet multipliers is given explicitly.  }

\section{Introduction}\label{sec:intro}
Pursuit and evasion problems have broad applications in areas such as autonomous vehicle control, modeling of predator--prey interactions, and space mission design. This paper gives a rigorous dynamical-systems perspective on how the geometry of the evader's path determines the qualitative behavior of pursuit.

We focus on the classical one-on-one pursuit and evasion problem in the plane. Pursuit problems have a long history dating back to Pierre Bouguer in 1732 \cite{Bouguer}; the work of A. S. Hathaway in 1921 \cite{Hathaway2}, Bernhart \cite{Bernhart} and many others has produced a rich body of results on pure pursuit curves. Most studies of pursuit and evasion focus on optimizing the strategies of the players: Isaacs' seminal work on differential games \cite{Isaacs} established the framework for analyzing optimal strategies, which has been extended to many-to-many scenarios \cite{Weintraub,Sergey}. A related line of research constrains the evader to a given curve: Azamov \cite{Azamov} studied escape strategies on prescribed curves, and Kuchkarov \cite{Kuchkarov} derived necessary and sufficient conditions for pursuit completion when the evader moves along a given curve. For pure pursuit of a target on a circular arc, Von Moll et al.~\cite{VonMoll} approximated the intercept time of a faster pursuer, using the separation distance and the bearing angle as variables. These studies address whether and when capture occurs, not the long-term \emph{shape} of the pursuer's trajectory, which is the subject of this paper.

We consider the following three conditions:
\begin{enumerate}
\item The evader follows a prescribed path, unaffected by the pursuer; it does not stop or turn back.
\item The pursuer's velocity vector constantly points toward the evader's position.
\item The ratio $n$ of the pursuer's speed to the evader's speed is constant.
\end{enumerate}
Barton and Eliezer formulated this problem as a system of differential equations \cite{Barton}. For an evader on a circle, Hathaway had already found three properties \cite{Hathaway2}: if $n>1$ the pursuer captures the evader; if $n=1$ the pursuer orbits the same circle as the evader; and if $n<1$ it converges to the circle reduced by the factor $n$ (Fig.~\ref{Fig-circular}). These properties have been confirmed numerically \cite{Ohira,Nahin}. Rozas \cite{Rozas} gave a numerical procedure for pursuit curves of objects in uniformly accelerated motion and treated an elliptical evader path as an example.

For an evader on an ellipse, our simulations \cite{Sota,CCP2023pb,JSIAML} showed that the first two properties persist but the third does not: for $n<1$ the pursuer's trajectory converges to a closed curve that is \emph{not} the ellipse reduced by the factor $n$ (Fig.~\ref{Fig-Elliptical}). Since the shape of the pursuer's trajectory does not depend on how the evader's path is parametrized \cite{CCP2023pb}, the evader may be assumed to move with unit speed, and the problem reduces to a planar system for the separation distance $\rho$ and the angle $\zeta$ between the two velocity vectors \cite{JSIAML}. For a circle this system is autonomous and has an asymptotically stable equilibrium; for an ellipse it is non-autonomous with $\pi$-periodic coefficients and has no equilibrium. The aim of this paper is to explain the closed curve of Fig.~\ref{Fig-Elliptical} rigorously.

 \begin{figure}[h]
  \centering
  \includegraphics[width=0.40\columnwidth]{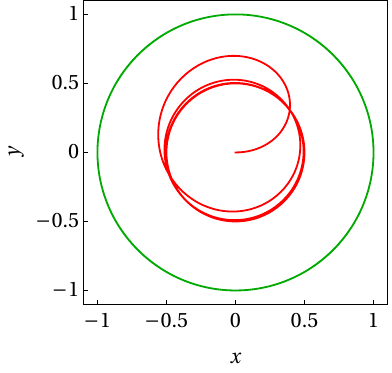}
  \caption{Circular one-on-one pursuit and evasion problem when $n=0.5$. The evader orbits a circle with radius one from (1, 0). The pursuer starts from the origin. The pursuer's trajectory converges to a circle with a radius $n=0.5$. The red and green lines are the pursuer's and evader's trajectories, respectively. Adapted from \cite[Fig.~1(a)]{CCP2023pb}, \copyright\ The Author(s) 2026, licensed under CC BY 4.0 (\url{https://creativecommons.org/licenses/by/4.0/}); axis labels redrawn}
  \label{Fig-circular}
 \end{figure}
  \begin{figure}[h]
  \centering
  \includegraphics[width=0.53\columnwidth]{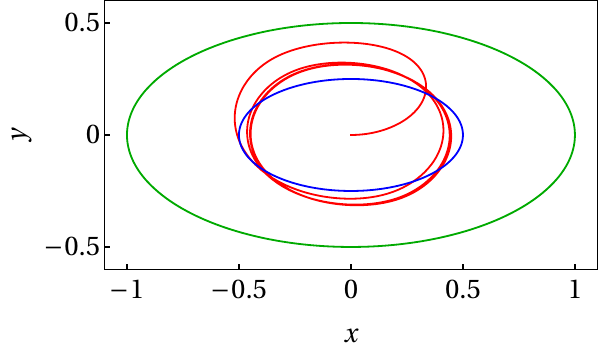}
 \caption{Elliptical one-on-one pursuit and evasion problem when $n=0.5$. The evader orbits the ellipse $X^2+Y^2/(0.5)^2=1$ from $(1,0)$, parametrized as $(X,Y)=(\cos t,0.5\sin t)$; the pursuer's trajectory is computed from the pursuit law \eqref{eq:pursuit}. The meaning of the red and green lines and other initial conditions are the same as in Fig. \ref{Fig-circular}. The pursuer's trajectory does not converge to the blue line, i.e., the ellipse orbited by the evader scaled by a factor of $n=0.5$. Adapted from \cite[Fig.~1(b)]{CCP2023pb}, \copyright\ The Author(s) 2026, licensed under CC BY 4.0 (\url{https://creativecommons.org/licenses/by/4.0/}); axis labels redrawn
 }
 \label{Fig-Elliptical}
\end{figure}

Our main results, for an evader on an ellipse with semi-axes $a\ge b>0$, are the following.
\begin{enumerate}
\item[(A)] If $n>1$, the pursuer captures the evader in finite time, with an explicit upper bound for the capture time in terms of $a$, $b$, $n$ and the initial distance (Theorem~\ref{thm:Finite-time}).
\item[(B)] If $0<n<1$, a pursuer starting in the closed elliptical disc bounded by the evader's path is never captured (Proposition~\ref{prop:nocapture}). Captured solutions nevertheless exist for every small initial separation: the pursuer then starts just outside the evader's path on a nearly head-on course, and we locate its initial angle asymptotically (Proposition~\ref{prop:capture}).
\item[(C)] If $0<n<1$, there is a unique $\pi$-periodic solution, and every solution that avoids capture converges to it, uniformly for pursuers starting in the elliptical disc (Theorem~\ref{thm:stability}). By (B), the restriction to solutions avoiding capture cannot be removed. In the circular case this solution is Hathaway's equilibrium; in the elliptical case it is the closed curve of Fig.~\ref{Fig-Elliptical}.
\item[(D)] The periodic solution is hyperbolic and exponentially stable, so that the convergence in (C) is eventually exponential; the product of its Floquet multipliers is given explicitly by an integral along the periodic solution (Proposition~\ref{prop:floquet}).
\end{enumerate}

The analysis rests on three devices. First, writing $\rho=e^{\mu}$ turns capture into a finite-time blow-up $\mu\to-\infty$, and the pair $(\mu,\zeta)$ combines into a complex variable $z=e^{\mu+i\zeta}$ that satisfies a single complex ODE; at each instant $z$ is related to the pursuer's position by an isometry, and the squared distance $|z_{1}-z_{2}|^{2}$ between two solutions is non-increasing. This non-expansiveness reflects the convexity of the Euclidean distance. On CAT$(0)$ spaces, which include the plane, Jun~\cite{Jun} obtained it from a contraction estimate for time-dependent gradient curves of convex functions and used it to prove uniqueness of pursuit curves for $n\ge 1$. Here the non-expansiveness is proved directly (Lemma~\ref{lem:contraction}) and is used to analyze the long-time behavior for $n<1$. Second, the ``depth'' of the pursuer inside the elliptical disc obeys an exact first-order equation, which yields the forward invariance of the disc and excludes capture from inside. Third, a half-period map on the disc is contractive in the sense of Edelstein \cite{Edelstein}, so that compactness gives existence, uniqueness and attractivity of the periodic solution at once. For uniformly contracting systems, the existence, uniqueness and attractivity of a periodic response are established in contraction theory, which was developed in control engineering~\cite{Lohmiller,Pavlov,Margaliot}. In our case, the contraction rate vanishes for radially separated solutions and the vector field is singular at $z=0$, so a uniform rate as in those works is not available; the compactness of the disc is used instead. All results are checked against the simulations of \cite{JSIAML}.

The rest of this paper is organized as follows. Section~\ref{sec:preliminaries} recalls the formulation of Barton and Eliezer and the independence of the pursuer's trajectory from the evader's parametrization. Section~\ref{sec:alternative} derives the dynamical system, fixes the notation, and gives its complex and fixed-plane formulations. Section~\ref{sec:analysis} proves finite-time capture for $n>1$ (Section~\ref{sec:blowup}), determines for $n<1$ which initial data can be captured (Section~\ref{sec:capture}), and establishes the existence, uniqueness, attractivity and exponential stability of the periodic solution (Section~\ref{sec:periodic}). Section~\ref{sec:verification} compares these results with numerical simulations.

\section{Formulation}\label{sec:preliminaries}
We formulate the problem following \cite{Barton}. Let ${\mathbf{E}}(t)=(X(t),Y(t))$ and ${\mathbf{P}}(t)=(x(t),y(t))$ denote the positions of the evader and the pursuer at time $t$, and let a dot denote $d/dt$. By the first condition of Section~\ref{sec:intro}, $\mathbf{E}$ is given and $\dot{\mathbf{E}}\ne \mathbf{0}$. The second and third conditions read
\begin{align}
\lambda \dot{\mathbf{P}}={\mathbf{E}}-{\mathbf{P}}, \quad \lambda(t)\ge 0,\label{eq-muki}
\qquad
|\dot{\mathbf{P}}|=n|\dot{\mathbf{E}}|,
\end{align}
where $n>0$ is the ratio of the pursuer's speed to the evader's speed and
\begin{equation}\label{eq:lambda}
  \lambda=\frac{|\mathbf{E}-\mathbf{P}|}{n|\dot{\mathbf{E}}|}
\end{equation}
is the time the pursuer would need to reach the evader's current position. Eliminating $\lambda$ gives the pursuit law
\begin{equation}\label{eq:pursuit}
  \dot{\mathbf{P}}=n|\dot{\mathbf{E}}|\,\frac{\mathbf{E}-\mathbf{P}}{|\mathbf{E}-\mathbf{P}|},
  \qquad \mathbf{P}\ne\mathbf{E},
\end{equation}
which determines $\mathbf{P}$ for given $\mathbf{E}$ and $n$ up to the moment of capture, where its right-hand side becomes singular.

The shape of the pursuer's trajectory does not depend on how the evader's path is parametrized. Indeed, conditions \eqref{eq-muki} are invariant under orientation-preserving reparametrization: if $u=u(t)$ with $du/dt>0$ and $(\mathbf{P}(t),\lambda(t))$ satisfies \eqref{eq-muki} for $\mathbf{E}(t)$, then, by the chain rule, $\mathbf{P}(t(u))$ satisfies them for $\mathbf{E}(t(u))$ with $\lambda$ replaced by $\lambda\,du/dt$. For the system \eqref{eq:pursuit} this was shown in \cite[Theorem~1]{CCP2023pb}. We may therefore assume that the evader moves with unit speed, $|\dot{\mathbf{E}}|\equiv1$, whenever only the shape of the pursuer's trajectory is concerned.

\section{Derivation of the Dynamical System}\label{sec:alternative}
\subsection{The dynamical system}\label{sec:derivation}
Assume $|\dot{\mathbf{E}}|\equiv1$ and identify $\mathbb{R}^{2}$ with $\mathbb{C}$. Write
\begin{equation}\label{DDS-1}
  \dot{\mathbf{E}}=e^{i\varphi(t)},\qquad \dot{\mathbf{P}}=ne^{i\theta(t)},\qquad
  \rho:=|\mathbf{E}-\mathbf{P}|,\qquad \zeta:=\varphi-\theta ,
\end{equation}
so that $\varphi$ and $\theta$ are the directions of motion of the two players and $\zeta$ is the angle between their velocity vectors. By \eqref{eq-muki}, $\mathbf{E}-\mathbf{P}=\rho e^{i\theta}$ and, by \eqref{eq:lambda},
\begin{equation}\label{eq:lambdarho}
  \lambda=\frac{\rho}{n}.
\end{equation}
Differentiating $\mathbf{E}-\mathbf{P}=\rho e^{i\theta}$ gives $(\dot\rho+i\rho\dot\theta)e^{i\theta}=e^{i\varphi}-ne^{i\theta}$, that is, $\dot\rho+i\rho\dot\theta=e^{i\zeta}-n$. Hence $\dot\rho=\cos\zeta-n$ and $\rho\dot\theta=\sin\zeta$, and in the logarithmic distance $\mu:=\log\rho$, i.e.\ $\rho=e^{\mu}$,
\begin{align}
    \dot{\mu}=e^{-\mu}(\cos\zeta-n),\quad \dot{\zeta}=-e^{-\mu}\sin\zeta+\dot{\varphi}.\label{eq:JSIAM2025-5}
\end{align}

This paper focuses on an evader orbiting counterclockwise the ellipse
\begin{align}
    \frac{X^2}{a^2}+\frac{Y^2}{b^2}=1. \label{EllipseDef}
\end{align}
Since the evader has unit speed and $\varphi$ is the angle of its tangent, $\dot\varphi$ is the curvature of the ellipse at $\mathbf{E}$, and $\varphi$ may be used to parametrize the ellipse. The outward unit normal at $\mathbf{E}$ is $(\sin\varphi,-\cos\varphi)$, and it is parallel to the gradient $(X/a^{2},Y/b^{2})$ of the left-hand side of \eqref{EllipseDef}; normalizing by \eqref{EllipseDef} gives
\begin{equation}\label{eq:E}
  \mathbf{E}(\varphi)=\frac{\bigl(a^{2}\sin\varphi,\,-b^{2}\cos\varphi\bigr)}{\sqrt{a^2\sin^2{\varphi}+b^2\cos^2{\varphi}}},
\end{equation}
and the curvature of the ellipse at this point is
\begin{equation}
\dot{\varphi}=f(\varphi):=\frac{(a^2\sin^2{\varphi}+b^2\cos^2{\varphi})^{3/2}}{a^2b^2}
\label{phi2t}
\end{equation}
(for instance $f(\pi/2)=a/b^{2}$ at $(a,0)$ and $f(0)=b/a^{2}$ at $(0,-b)$). Since $\dot\varphi>0$, the parameter $t$ can be replaced by $\varphi$. Dividing \eqref{eq:JSIAM2025-5} by $\dot\varphi=f(\varphi)$, we obtain the pursuit-evasion dynamical system in the parameter $\varphi$:
\begin{simultempheq}{align}{equation-rev}
&\frac{d\mu}{d\varphi}=\frac{e^{-\mu(\varphi)}(\cos \zeta(\varphi)-n)}{f(\varphi)}, \label{rev-mu}\\
&\frac{d\zeta}{d\varphi}=-\frac{e^{-\mu(\varphi)}\sin \zeta(\varphi)}{f(\varphi)}+1.\label{rev-zeta}
\end{simultempheq}
Let the initial value of $\varphi$ be $\varphi_0$. For instance, if the evader starts at $(a, 0)$ and revolves counterclockwise around the ellipse, $\varphi_0 = \pi /2$.

\begin{remark}\label{rem:mu}
In the variables $(\rho,\zeta)$ of \cite{JSIAML}, numerical solutions for $n>1$ continue past the moment of capture with $\rho<0$ \cite[Figs.~2 and~5]{JSIAML}. This is an artifact of the coordinates rather than of the model: a signed distance identifies $(-\rho,\zeta)$ with $(\rho,\zeta+\pi)$, and in the complex variable $z=\rho e^{i\zeta}$ of Section~\ref{sec:complex} such a solution simply passes through $z=0$. Physically, however, the pursuit ends at capture. Writing $\rho=e^{\mu}$ turns capture into a finite-time blow-up $\mu\to-\infty$, beyond which the solution is not continued. Restricting to $z\ne0$, or to $\rho>0$, would exclude this continuation equally well; the advantage of $\mu$ is that the phase space becomes $\mathbb{R}\times(\mathbb{R}/2\pi\mathbb{Z})$, so that the constraint $\rho>0$ holds automatically and capture is a blow-up in the usual sense. The variable $\mu$ serves mainly this purpose (Definitions~\ref{def:fates} and~\ref{def:blowup} and Theorem~\ref{thm:Finite-time}); the analysis of Sections~\ref{sec:capture} and~\ref{sec:periodic} is carried out in $\rho=|z|$ and in $z$.
\end{remark}

\subsection{Notation}\label{sec:reform}
From now on a prime denotes $d/d\varphi$, unless another independent variable is indicated explicitly; for a function of one variable such as $N$ below, the prime is its ordinary derivative. Since \eqref{equation-rev} depends on $\zeta$ only through $\cos\zeta$ and $\sin\zeta$, the angle $\zeta$ is determined only modulo $2\pi$; statements such as $\zeta\to\pi$ or $0\le\Delta\le\delta_{*}$ below refer to a suitable continuous lift of $\zeta$, and $z=e^{\mu+i\zeta}$ does not depend on this choice.

Throughout the rest of the paper we assume $a\ge b>0$; this is no loss of generality, since a rotation of the plane by $\pi/2$ exchanges the two semi-axes, and the circular case $a=b$ is included. Writing
\[
  N(x):=a^{2}\sin^{2}x+b^{2}\cos^{2}x,\qquad h(\varphi):=\sqrt{N(\varphi)},
\]
we have $f=h^{3}/(a^{2}b^{2})$, and we put
\begin{equation}\label{eq:alpha}
  \alpha:=\frac{1}{f},\qquad
  \alpha_{\min}:=\min\alpha=\frac{b^{2}}{a},\qquad
  \alpha_{\max}:=\max\alpha=\frac{a^{2}}{b},
\end{equation}
equivalently $f_{\min}=b/a^{2}$ and $f_{\max}=a/b^{2}$; thus $\alpha$ is smooth, positive and $\pi$-periodic. In this notation the evader's position \eqref{eq:E} reads $\mathbf{E}(\varphi)=(a^{2}\sin\varphi,\,-b^{2}\cos\varphi)/h(\varphi)$, and we set
\begin{equation}\label{eq:Q}
  Q:=\begin{pmatrix}a^{-2}&0\\[2pt]0&b^{-2}\end{pmatrix},\quad
  \langle\mathbf{p}_{1},\mathbf{p}_{2}\rangle_{Q}
    :=\langle\mathbf{p}_{1},Q\mathbf{p}_{2}\rangle,\quad
  q(\mathbf{p}):=\langle\mathbf{p},\mathbf{p}\rangle_{Q},\quad
  K:=\{q\le1\},
\end{equation}
so that $K$ is the closed elliptical disc bounded by the evader's path: it is compact and convex, $K=-K$ and $\operatorname{diam}K=2a$. We write $\mathbf{P}=(x,y)$ for the pursuer's position, $\theta=\varphi-\zeta$ for the direction of its velocity as in Section~\ref{sec:derivation}, and $\boldsymbol{\nu}:=(\sin\varphi,-\cos\varphi)=-ie^{i\varphi}$ for the outward unit normal of $\partial K$ at $\mathbf{E}(\varphi)$; note that $q(e^{i\theta})=N(\theta)/(a^{2}b^{2})\in[a^{-2},b^{-2}]$.

Sections~\ref{sec:capture} and~\ref{sec:periodic} use exactly three properties of the ellipse, all immediate from \eqref{eq:E}:
\begin{equation}\label{eq:ident}
  q(\mathbf{E})\equiv1,\qquad
  \mathbf{E}'(\varphi)=\alpha(\varphi)\,e^{i\varphi},\qquad
  Q\mathbf{E}=\frac{\boldsymbol{\nu}}{h},
\end{equation}
the second being \eqref{phi2t} rewritten in the parameter $\varphi$, together with the consequences
\begin{equation}\label{eq:EQtheta}
  \langle\mathbf{E},e^{i\varphi}\rangle_{Q}=0,\qquad
  \langle\mathbf{E},e^{i\theta}\rangle_{Q}
     =\frac{\langle e^{i\theta},\boldsymbol{\nu}\rangle}{h}
     =\frac{\sin\zeta}{h},\qquad
  \mathbf{E}(\varphi+\pi)=-\mathbf{E}(\varphi).
\end{equation}
The middle identity of \eqref{eq:EQtheta} is the one that does the work: it converts the algebraic condition $\mathbf{P}\in K$ into a statement about the angle $\zeta$, and with it the geometric part of Section~\ref{sec:analysis} reduces to the single scalar equation of Lemma~\ref{lem:eta}.

\subsection{Complex formulation}\label{sec:complex}
Define $z(\varphi):=e^{\mu(\varphi)+i\zeta(\varphi)}$, so that $|z|=e^{\mu}=\rho$ and $z/|z|=e^{i\zeta}$. From \eqref{equation-rev},
\begin{align}
    \frac{dz}{d\varphi}
    &= e^{\mu+i\zeta}\!\left(\frac{d\mu}{d\varphi}+i\frac{d\zeta}{d\varphi}\right)
     = e^{\mu+i\zeta}\!\left(
         \frac{e^{-\mu}(e^{-i\zeta}-n)}{f(\varphi)}+i
       \right)
     = \frac{1}{f(\varphi)}-\frac{n}{f(\varphi)}\frac{z}{|z|}+iz.
    \label{equation-rev-2}
\end{align}
In terms of $\rho=e^{\mu}$, equations \eqref{rev-mu}--\eqref{rev-zeta} read
\begin{equation}\label{eq:rhozeta}
  \rho'=\alpha(\varphi)\bigl(\cos\zeta-n\bigr),\qquad
  \zeta'=1-\frac{\alpha(\varphi)\sin\zeta}{\rho}.
\end{equation}
The right-hand side of \eqref{equation-rev-2} is smooth on $\{\rho>0\}$ and grows at most linearly in $|z|$, so $|z|$ cannot become unbounded in finite $\varphi$ and a maximal solution can cease to exist only by approaching $z=0$. This dichotomy is what the following terminology records.

\begin{definition}\label{def:fates}
Let $z$ be the maximal solution of \eqref{equation-rev-2} with $z(\varphi_{0})=w$, defined on $[\varphi_{0},\varphi_{\max})$. We call $z$ \emph{captured} if $\varphi_{\max}<\infty$ and $|z(\varphi)|\to0$ (equivalently $\mu\to-\infty$) as $\varphi\uparrow\varphi_{\max}$, and \emph{persistently evading} if $\varphi_{\max}=\infty$ and $\liminf_{\varphi\to\infty}|z(\varphi)|>0$.
\end{definition}

Here ``evasion'' refers to the outcome, not to a strategy of the evader. We do not claim that the two alternatives are exhaustive; nothing below requires solutions with $\varphi_{\max}=\infty$ and $\liminf|z|=0$ to be excluded.

Throughout, $\varphi_{0}$ being the initial value of $\varphi$, we put $K_{0}:=K\setminus\{\mathbf{E}(\varphi_{0})\}$: the single point $\mathbf{E}(\varphi_{0})$ has to be excluded from the admissible initial data because it corresponds to $z(\varphi_{0})=0$, which does not lie in the domain $\mathbb{C}\setminus\{0\}$ of \eqref{equation-rev-2}.

\begin{remark}\label{rem:noeq}
Equation \eqref{equation-rev-2} has a constant solution only in the circular case. Indeed, if $z\equiv w\ne0$ is a constant solution, then \eqref{equation-rev-2} gives $\alpha(\varphi)\bigl(1-n\,w/|w|\bigr)=-iw$ for every $\varphi$. The factor $1-n\,w/|w|$ does not vanish, since otherwise $w=0$; hence $\alpha$ is constant, which by \eqref{phi2t} and \eqref{eq:alpha} happens only for $a=b$. Conversely, if $a=b$ then $\alpha\equiv a$, and writing $w=\rho e^{i\zeta}$ the condition becomes $a(e^{-i\zeta}-n)=-i\rho$, that is, $\cos\zeta=n$ and $\rho=a\sin\zeta$. This has a solution with $\rho>0$ exactly when $0<n<1$, namely Hathaway's equilibrium $\rho=a\sqrt{1-n^{2}}$, $\zeta=\arccos n$. In particular, for an elliptical path ($a>b$) the system \eqref{equation-rev} has no equilibrium.
\end{remark}

\subsection{The pursuit law in the fixed plane}\label{sec:fixedplane}
With $|\dot{\mathbf{E}}|\equiv1$, the pursuit law \eqref{eq:pursuit} says that the pursuer moves straight at the evader with speed $n$; dividing by $\dot\varphi=f(\varphi)$, the same law in the parameter $\varphi$ reads
\begin{equation}\label{eq:veq}
  \frac{d\mathbf{P}}{d\varphi}
  =n\alpha(\varphi)\,\frac{\mathbf{E}(\varphi)-\mathbf{P}}{|\mathbf{E}(\varphi)-\mathbf{P}|}
  =n\alpha(\varphi)\,e^{i\theta},
  \qquad \mathbf{P}\ne\mathbf{E}(\varphi).
\end{equation}
That \eqref{eq:veq} and \eqref{equation-rev-2} describe the same motion is only a change of unknown. What is used repeatedly below are the last two identities of \eqref{eq:isom}: the change of unknown is an \emph{isometry}, and it exchanges $\pi$-periodicity of $z$ with $\pi$-antiperiodicity of $\mathbf{P}$. The first lets the contraction estimate of Lemma~\ref{lem:contraction}, stated for $z$, be read as a statement about pursuers; the second is what turns a half-period map on $K$ into a device for producing $\pi$-periodic solutions.

\begin{proposition}\label{prop:dict}
Along any solution,
\begin{equation}\label{eq:dict}
  \mathbf{E}(\varphi)-\mathbf{P}(\varphi)=\rho\,e^{i\theta}
  =e^{i\varphi}\overline{z(\varphi)} ,
\end{equation}
$z$ solves \eqref{equation-rev-2} on an interval if and only if $\mathbf{P}$ solves \eqref{eq:veq} there, and
\begin{equation}\label{eq:isom}
  |\mathbf{E}-\mathbf{P}|=|z|=\rho,\qquad
  |z_{1}-z_{2}|=|\mathbf{P}_{1}-\mathbf{P}_{2}|,\qquad
  z\ \text{is $\pi$-periodic}\iff
  \mathbf{P}(\varphi+\pi)=-\mathbf{P}(\varphi).
\end{equation}
\end{proposition}

\begin{proof}
By pure pursuit $\mathbf{E}-\mathbf{P}$ is the positive multiple $\rho e^{i\theta}$ of the pursuer's velocity direction, while $z=\rho e^{i\zeta}$ with $\zeta=\varphi-\theta$; hence $e^{i\varphi}\bar z=\rho e^{i(\varphi-\zeta)}=\rho e^{i\theta}$, which is \eqref{eq:dict}. Conjugating \eqref{equation-rev-2} gives $\bar z'=\alpha-n\alpha\bar z/|z|-i\bar z$, so, by \eqref{eq:ident},
\[
  \bigl(e^{i\varphi}\bar z\bigr)'
  =ie^{i\varphi}\bar z+e^{i\varphi}\bar z'
  =\alpha e^{i\varphi}-n\alpha e^{i\theta}
  =\mathbf{E}'-n\alpha e^{i\theta},
\]
which, with $\mathbf{P}=\mathbf{E}-e^{i\varphi}\bar z$, is \eqref{eq:veq}; the steps are reversible. Finally, for each fixed $\varphi$ the map $z\mapsto e^{i\varphi}\bar z$ is an $\mathbb{R}$-linear isometry of $\mathbb{C}$, which gives the first two identities of \eqref{eq:isom}; the third follows from $e^{i(\varphi+\pi)}=-e^{i\varphi}$ together with $\mathbf{E}(\varphi+\pi)=-\mathbf{E}(\varphi)$.
\end{proof}

\section{Analysis of Dynamical System}\label{sec:analysis}

\subsection{Blow-up Time for $n>1$}\label{sec:blowup}
If $n>1$, the distance $\rho$ between the players reaches $0$ in finite $\varphi$, so that $\mu=\log\rho$ diverges to $-\infty$, i.e.\ the solution of \eqref{equation-rev} blows up. In this subsection we derive an upper bound for the blow-up time.

\begin{definition}\label{def:blowup}
The blow-up time $\varphi_B$ is the value at which $\mu(\varphi) \to -\infty$ as $\varphi \to \varphi_B - 0$.
\end{definition}

\begin{theorem}[Finite-time Blow-up]
Assume $a \ge b > 0$ and $n > 1$. Then every solution of \eqref{equation-rev} blows up in finite time, and the blow-up time obeys
\begin{align}
    \varphi_B  \le \frac{a e^{\mu(\varphi_0)}}{b^2 (n - 1)}+ \varphi_0 = \frac{a \rho(\varphi_0)}{b^2 (n - 1)}+ \varphi_0.
\end{align}
\label{thm:Finite-time}
\end{theorem}
\begin{proof}
Let $[\varphi_{0},\varphi_{\max})$ be the maximal interval of existence of the solution of \eqref{equation-rev}; on it $\mu$ is finite, so $e^{\mu}>0$. By \eqref{rev-mu} and \eqref{eq:alpha},
\begin{equation}\label{eq:blowup_ineq}
  \bigl(e^{\mu}\bigr)'=e^{\mu}\mu'=\alpha(\cos\zeta-n)
  \le-(n-1)\alpha_{\min}=-\frac{(n-1)b^{2}}{a},
\end{equation}
so that $0<e^{\mu(\varphi)}\le e^{\mu(\varphi_{0})}-(n-1)b^{2}(\varphi-\varphi_{0})/a$ on $[\varphi_{0},\varphi_{\max})$. This forces $\varphi_{\max}\le\varphi_{0}+a e^{\mu(\varphi_{0})}/(b^{2}(n-1))$. Since a maximal solution can end only by approaching $z=0$ (Section~\ref{sec:complex}) and $\mu$ is decreasing, $\mu\to-\infty$ as $\varphi\uparrow\varphi_{\max}=:\varphi_{B}$.
\end{proof}
Note that \eqref{eq:blowup_ineq} both \emph{produces} the blow-up and bounds its time; no prior assumption that $\varphi_B$ exists is needed. Since $|\dot{\mathbf{E}}|\equiv 1$, the evader's path is parametrized by arc length. Thus $f$---the time derivative of the angle of the evader's velocity vector---is the curvature of the evader's path, and $\alpha=1/f$ is its radius of curvature. In the time variable, \eqref{eq:blowup_ineq} is the estimate $\dot\rho=\cos\zeta-n\le-(n-1)$, which bounds the capture time by $\rho(\varphi_0)/(n-1)$ for any evader path; Theorem~\ref{thm:Finite-time} expresses this bound in the parameter $\varphi$ through the maximum curvature $f_{\max}=a/b^{2}$ of the evader's path.

\subsection{Capture for $n<1$}\label{sec:capture}

Throughout this subsection and Section~\ref{sec:periodic}, $a\ge b>0$ and $0<n<1$. This subsection determines which initial data can lead to capture: capture must be head-on (Lemma~\ref{lem:headon}); it is therefore impossible for a pursuer starting in the closed elliptical disc $K$ (Proposition~\ref{prop:nocapture}); and it does occur for suitable pursuers starting just outside $K$ (Proposition~\ref{prop:capture}).

\subsubsection{The depth of the pursuer inside the ellipse}

Everything needed below about the position of the pursuer relative to $K$ is carried by one algebraic quantity, the ``depth''
\begin{equation}\label{eq:eta}
  \eta:=1-q(\mathbf{P})\qquad(\mathbf{P}\in K\iff\eta\ge0),
\end{equation}
and its normalization $\Omega:=\eta/\rho$. Being a polynomial in $\mathbf{P}$, $\eta$ is smooth across $\partial K$ and obeys an exact first-order equation; this is what makes any discussion of the distance to $K$, which is not differentiable on $\partial K$, unnecessary.

\begin{lemma}[two exact identities]\label{lem:eta}
Along any solution of \eqref{eq:veq} with $\rho>0$,
\begin{equation}\label{eq:etaform}
  \eta=\frac{2\rho\sin\zeta}{h(\varphi)}-\rho^{2}q(e^{i\theta}),
  \qquad\text{so}\qquad
  \Omega=\frac{2\sin\zeta}{h(\varphi)}-\rho\,q(e^{i\theta})\ \le\ \frac{2}{b},
\end{equation}
and
\begin{equation}\label{eq:etaODE}
  \eta'=\frac{n\alpha}{\rho}\bigl(\rho^{2}q(e^{i\theta})-\eta\bigr),
  \qquad
  \Omega'=n\alpha\,q(e^{i\theta})-\frac{\alpha\cos\zeta}{\rho}\,\Omega .
\end{equation}
\end{lemma}

\begin{proof}
Since $\mathbf{P}=\mathbf{E}-\rho e^{i\theta}$ and $q(\mathbf{E})=1$,
\[
  \eta=1-q(\mathbf{E})+2\rho\langle \mathbf{E},e^{i\theta}\rangle_{Q}-\rho^{2}q(e^{i\theta})
     =2\rho\langle \mathbf{E},e^{i\theta}\rangle_{Q}-\rho^{2}q(e^{i\theta}) ,
\]
which is the first identity by \eqref{eq:EQtheta}; and $\Omega\le 2/h\le2/b$ because $\sin\zeta\le1$, $q(e^{i\theta})>0$ and $h\ge b$.

For the second, differentiate $\eta=1-q(\mathbf{P})$ and use \eqref{eq:veq}:
\[
  \eta'=-2\langle\mathbf{P},\mathbf{P}'\rangle_{Q}
       =-2n\alpha\bigl\langle \mathbf{P},e^{i\theta}\bigr\rangle_{Q}.
\]
Now $\langle \mathbf{P},e^{i\theta}\rangle_{Q} =\langle \mathbf{E},e^{i\theta}\rangle_{Q}-\rho\,q(e^{i\theta}) =\dfrac{\eta-\rho^{2}q(e^{i\theta})}{2\rho}$, which gives the first equation of \eqref{eq:etaODE}. Finally, with $\Omega=\eta/\rho$ and $n\alpha+\rho'=\alpha\cos\zeta$ by \eqref{eq:rhozeta},
\[
  \Omega'=\frac{\eta'}{\rho}-\frac{\eta\rho'}{\rho^{2}}
  =n\alpha q(e^{i\theta})-\frac{\Omega}{\rho}\bigl(n\alpha+\rho'\bigr)
  =n\alpha q(e^{i\theta})-\frac{\alpha\cos\zeta}{\rho}\Omega. \qedhere
\]
\end{proof}

The geometry of $K$ along solutions is now a two-line consequence.

\begin{corollary}[the elliptical disc along solutions]\label{cor:Kgeom}
Let $\mathbf{P}$ solve \eqref{eq:veq} with $\rho>0$ on an interval, and put $\psi:=q(\mathbf{P})-1=-\eta$. Then $\psi'<0$ at every point where $\psi\ge0$, and consequently:
\begin{enumerate}
\item[\textup{(a)}] $q(\mathbf{P}(\varphi))\le \max\{q(\mathbf{P}(\varphi_{1})),1\}$ for every $\varphi\ge\varphi_{1}$ in the interval; in particular $K$ is forward invariant;
\item[\textup{(b)}] if $\mathbf{P}(\varphi_{1})\in K$ then $\mathbf{P}(\varphi)\in\operatorname{int}K$ for every $\varphi>\varphi_{1}$;
\item[\textup{(c)}] the pursuer of any $\pi$-periodic solution of \eqref{equation-rev-2} satisfies $q(\mathbf{P})<1$, that is, it stays in the open elliptical disc and never meets the evader's own curve.
\end{enumerate}
\end{corollary}

\begin{proof}
By \eqref{eq:etaODE},
\begin{equation}\label{eq:psiODE}
  \psi'=-\eta'=-\frac{n\alpha}{\rho}\bigl(q(e^{i\theta})\rho^{2}+\psi\bigr),
\end{equation}
so $\psi\ge0$ forces $q(e^{i\theta})\rho^{2}+\psi>0$ and hence $\psi'<0$.

(a) Since $\psi'<0$ wherever $\psi\ge0$, a solution cannot cross a level $c\ge0$ of $\psi$ upwards, so $\{\psi\le c\}$ is forward invariant for every $c\ge0$. Taking $c:=\max\{q(\mathbf{P}(\varphi_{1}))-1,0\}$ gives the stated bound, and $c=0$ the invariance of $K=\{\psi\le0\}$.

(b) By (a), $\psi\le0$ after $\varphi_{1}$; a zero of $\psi$ at some $\varphi_{2}>\varphi_{1}$ would force $\psi'(\varphi_{2})<0$, hence $\psi>0$ just before $\varphi_{2}$.

(c) If $z$ is $\pi$-periodic then $\mathbf{P}(\varphi+\pi)=-\mathbf{P}(\varphi)$ by \eqref{eq:isom}, and $q(-\mathbf{p})=q(\mathbf{p})$, so $\psi$ is continuous and $\pi$-periodic and attains its maximum $\psi_{\max}$ at some $\varphi_{m}$, where $\psi'(\varphi_{m})=0$. If $\psi_{\max}\ge0$ this contradicts $\psi'(\varphi_{m})<0$; hence $\psi_{\max}<0$, i.e.\ $q(\mathbf{P})<1$.
\end{proof}

\subsubsection{No capture from inside the disc}

This is the only step of the proof of Theorem~\ref{thm:stability} that needs a genuinely quantitative argument. Geometrically, Lemma~\ref{lem:headon} shows that a capture from inside $K$ can be neither a tail chase, in which the slower pursuer falls behind, nor an approach from the side, in which the line of sight would have to turn through an infinite angle; only the head-on configuration remains. We use repeatedly that, by \eqref{eq:rhozeta}, every solution satisfies $|\rho'|\le C_{1}:=\alpha_{\max}(1+n)$; in particular, if $\rho(\varphi)\to0$ as $\varphi\uparrow\varphi_{c}$, then
\begin{equation}\label{eq:rholin}
  \rho(\varphi)\le C_{1}(\varphi_{c}-\varphi)\qquad(\varphi<\varphi_{c}).
\end{equation}

\begin{lemma}[capture is head-on]\label{lem:headon}
Let $\mathbf{P}$ solve \eqref{eq:veq} on $[\varphi_{1},\varphi_{c})$ with $\varphi_{c}<\infty$, $\mathbf{P}(\varphi)\in K$ and $\rho(\varphi)\to0$ as $\varphi\uparrow\varphi_{c}$. Then $\zeta(\varphi)\to\pi$.
\end{lemma}

\begin{proof}
By \eqref{eq:etaform} and $\eta\ge0$ (Corollary~\ref{cor:Kgeom}(a)) we have $\sin\zeta\ge\tfrac12 h\rho\,q(e^{i\theta})>0$, so, taking the continuous branch of $\zeta$ with $\zeta(\varphi_{1})\in(0,\pi)$, we have $\zeta\in(0,\pi)$ throughout; that is, the pursuer stays on the inner side of the tangent line to $\partial K$ at the evader. Since $0<n<1$, the angle $\delta_{0}:=\arccos n$ lies in $(0,\tfrac{\pi}{2})$, and $(0,\pi)$ splits into the \emph{tail-chase zone} $(0,\delta_{0})$, where $\cos\zeta>n$ and hence $\rho'>0$ by \eqref{eq:rhozeta}, and its complement $[\delta_{0},\pi)$. The tail-chase zone is non-empty precisely because $n<1$.

Fix $\delta\in(0,\delta_{0}]$; since $\delta<\pi/2$, the band $B_{\delta}:=[\delta,\pi-\delta]$ is a non-empty subinterval of $(0,\pi)$. As $\rho\to0$, choose $\varphi_{2}$ with $\rho<\tfrac12\alpha_{\min}\sin\delta$ on $[\varphi_{2},\varphi_{c})$. On $B_{\delta}$ we have $\sin\zeta\ge\sin\delta$, so $\zeta'=1-\alpha\sin\zeta/\rho\le1-2=-1$ by \eqref{eq:rhozeta}; in particular, after $\varphi_{2}$, $\zeta$ cannot cross either end of $B_{\delta}$ upwards.

(i) \emph{No tail chase:} $\zeta\ge\delta$ on $[\varphi_{2},\varphi_{c})$. Otherwise $\zeta(\varphi_{3})<\delta$ for some $\varphi_{3}\ge\varphi_{2}$, and then $\zeta<\delta\le\delta_{0}$ on $[\varphi_{3},\varphi_{c})$, since the lower end of $B_{\delta}$ cannot be crossed upwards. Thus $\zeta$ stays in the tail-chase zone and $\rho$ increases there, contradicting $\rho\to0$.

(ii) \emph{No approach from the side:} $\zeta>\pi-\delta$ on $[\varphi_{2},\varphi_{c})$. Otherwise $\zeta(\varphi_{4})\le\pi-\delta$ for some $\varphi_{4}\ge\varphi_{2}$, and by (i) and the upper end of $B_{\delta}$, $\zeta\in B_{\delta}$ on $[\varphi_{4},\varphi_{c})$. There the direction $\theta=\varphi-\zeta$ of the line of sight turns at the rate $\theta'=1-\zeta'=\alpha\sin\zeta/\rho\ge\alpha_{\min}\sin\delta/\rho$. By \eqref{eq:rholin},
\[
  \int_{\varphi_{4}}^{\varphi_{c}}\theta'\,d\varphi
  \ \ge\ \frac{\alpha_{\min}\sin\delta}{C_{1}}
         \int_{\varphi_{4}}^{\varphi_{c}}\frac{d\varphi}{\varphi_{c}-\varphi}
  \ =\ +\infty,
\]
whereas $\theta=\varphi-\zeta$ is bounded on the bounded interval $[\varphi_{4},\varphi_{c})$ because $\zeta\in(0,\pi)$.

Hence $\pi-\delta<\zeta<\pi$ on $[\varphi_{2},\varphi_{c})$. As $\delta\in(0,\delta_{0}]$ was arbitrary, $\zeta\to\pi$.
\end{proof}

\begin{proposition}[capture requires the pursuer to be outside $K$]\label{prop:nocapture}
Let $\mathbf{p}\in K$ with $\mathbf{p}\ne\mathbf{E}(\varphi_{0})$. Then the solution of \eqref{eq:veq} with $\mathbf{P}(\varphi_{0})=\mathbf{p}$ exists on $[\varphi_{0},\infty)$ and satisfies $\mathbf{P}(\varphi)\in K$ and $\rho(\varphi)>0$ for all $\varphi\ge\varphi_{0}$. In particular, the corresponding solution of \eqref{equation-rev-2} is not captured.
\end{proposition}

\begin{proof}
Let $[\varphi_{0},\varphi_{c})$ be the maximal interval of existence. By Corollary~\ref{cor:Kgeom}(a), $\mathbf{P}$ stays in $K$, so $\rho\le\operatorname{diam}K=2a$. If $\varphi_{c}<\infty$, the solution must approach $z=0$ (Section~\ref{sec:complex}); since $\rho$ is Lipschitz, it has a limit at $\varphi_{c}$, which is therefore $0$. Assume this and derive a contradiction.

By Lemma~\ref{lem:headon}, $\cos\zeta\le-\tfrac12$ on some $[\varphi_{4},\varphi_{c})$. There $\Omega=\eta/\rho\ge0$ and, by \eqref{eq:etaODE},
\begin{equation}\label{eq:Omgrow}
  \Omega'=n\alpha q(e^{i\theta})-\frac{\alpha\cos\zeta}{\rho}\Omega
  \ \ge\ \frac{n\alpha_{\min}}{a^{2}}+\frac{\alpha_{\min}}{2\rho}\,\Omega
  \ \ge\ \frac{n\alpha_{\min}}{a^{2}}>0,
\end{equation}
using $q(e^{i\theta})\ge|e^{i\theta}|^{2}/a^{2}=1/a^{2}$. Hence $\Omega(\varphi_{5})>0$ for some $\varphi_{5}\in(\varphi_{4},\varphi_{c})$, and then $\Omega>0$ afterwards, so \eqref{eq:Omgrow} gives $(\log\Omega)'\ge\alpha_{\min}/(2\rho)$ on $[\varphi_{5},\varphi_{c})$, so by \eqref{eq:rholin}
\[
  \log\Omega(\varphi)\ \ge\ \log\Omega(\varphi_{5})
  +\frac{\alpha_{\min}}{2C_{1}}\log\frac{\varphi_{c}-\varphi_{5}}{\varphi_{c}-\varphi}
  \ \xrightarrow[\varphi\uparrow\varphi_{c}]{}\ +\infty .
\]
This contradicts the a priori bound $\Omega\le2/b$ of \eqref{eq:etaform}. Hence $\varphi_{c}=\infty$ and $\rho>0$ throughout.
\end{proof}

\subsubsection{Existence of captured solutions}\label{sec:captured}

Proposition~\ref{prop:nocapture} shows that capture never starts inside $K$. We now show that it does occur from outside: for $n<1$ captured solutions exist for every small initial separation, and we locate the corresponding initial angle asymptotically. The whole argument rests on the single identity \eqref{eq:etaODE}, read in a normalization that turns it into a linear equation; no further geometry is needed. Lemma~\ref{lem:corner} shows that, in the corner $\{0\le\Delta\le\delta_{*}\}$ defined below, the normalized depth $\chi$ is pushed away from a value close to $-n/(2+n)$. Proposition~\ref{prop:capture} uses this twice: solutions starting on either side of that value leave the corner, so a captured solution must start in between (existence), and it must stay close to that value throughout (asymptotics).

Along a solution with $\rho>0$ put
\begin{equation}\label{eq:chidef}
  \Delta:=\pi-\zeta,\qquad
  s:=f(\varphi)\rho,\qquad
  \chi:=\frac{\eta}{q(e^{i\theta})\,\rho^{2}} ,
\end{equation}
so that, by the first identity of \eqref{eq:etaform},
\begin{equation}\label{eq:chiform}
  \chi=\frac{2\sin\zeta}{h(\varphi)\,\rho\,q(e^{i\theta})}-1 .
\end{equation}
Thus $\chi$ is the depth $\eta$ measured in the only dimensionless unit available, and \eqref{eq:chiform} has two consequences used below: $\mathbf{P}\in K\iff\eta\ge0\iff\chi\ge0$, and $\chi\ge-1$ as long as $\zeta\in(0,\pi]$, with $\chi=-1$ exactly when $\zeta=\pi$. Finally we fix, once and for all, the corner threshold
\begin{equation}\label{eq:Dstar}
  C_{2}:=\frac{a^{2}-b^{2}}{b^{2}},\qquad
  \delta_{*}:=\min\Bigl\{\tfrac18,\ \frac{n}{64\,C_{2}}\Bigr\}
\end{equation}
(the second term being omitted when $a=b$), which depends only on $n$, $a$, $b$, and, in the corner $\{0\le\Delta\le\delta_{*}\}$, the coefficient
\begin{equation}\label{eq:kappa}
  \kappa:=\frac{n}{\cos\Delta+n}\in\Bigl[\frac{n}{1+n},\,
  \frac{n}{\cos\delta_{*}+n}\Bigr]\subset(0,1),
\end{equation}
the range following from $\cos\delta_{*}\le\cos\Delta\le1$.

\begin{lemma}[the depth equation in the corner is linear]\label{lem:corner}
Let $z$ be a solution with $\rho>0$ on an interval $[\varphi_{a},\varphi_{b})$ on which $0\le\Delta\le\delta_{*}$. Then $\rho$ is strictly decreasing there, so that
\begin{equation}\label{eq:sigma}
  \sigma:=\log\frac{\rho(\varphi_{a})}{\rho(\varphi)}
\end{equation}
may be used as the independent variable, and $\eta$, $\chi$ satisfy the \emph{linear} equations
\begin{equation}\label{eq:chiODE}
  \frac{d\eta}{d\sigma}=-\kappa\,\eta+\kappa\,q(e^{i\theta})\rho^{2},
  \qquad
  \frac{d\chi}{d\sigma}=\bigl(2-\kappa-\varepsilon\bigr)\chi+\kappa,
  \qquad
  \varepsilon:=\frac{d\log q(e^{i\theta})}{d\sigma},
\end{equation}
where $|\varepsilon|\le\tfrac1{64}$. In particular $2-\kappa-\varepsilon\ge\tfrac12>0$, and the second equation reads $d\chi/d\sigma=(2-\kappa-\varepsilon)(\chi-\chi_{\sharp})$ with $\chi_{\sharp}:=-\kappa/(2-\kappa-\varepsilon)$: the right-hand side vanishes at $\chi_{\sharp}$ and has slope at least $\tfrac12$ there, so $\chi$ is pushed away from $\chi_{\sharp}$. Since $\kappa$ and $\varepsilon$ vary along the solution, so does $\chi_{\sharp}$, but it always stays close to $-n/(2+n)$:
\begin{equation}\label{eq:chisharp}
  \Bigl|\chi_{\sharp}+\frac{n}{2+n}\Bigr|
  \ \le\ \delta_{*}^{2}+\frac{C_{2}\delta_{*}}{2n}\ \le\ \frac18 .
\end{equation}
\end{lemma}

\begin{proof}
Since $\Delta\le\delta_{*}\le\tfrac18$ we have $\cos\Delta+n\ge\cos\delta_{*}+n>n>0$, so by \eqref{eq:rhozeta} $\rho'=\alpha(\cos\zeta-n)=-\alpha(\cos\Delta+n)<0$ and $d\sigma/d\varphi=-\rho'/\rho=\alpha(\cos\Delta+n)/\rho>0$. Dividing the first equation of \eqref{eq:etaODE} by $d\sigma/d\varphi$ gives
\[
  \frac{d\eta}{d\sigma}
  =\frac{n\alpha\bigl(q(e^{i\theta})\rho^{2}-\eta\bigr)/\rho}
        {\alpha(\cos\Delta+n)/\rho}
  =\kappa\bigl(q(e^{i\theta})\rho^{2}-\eta\bigr),
\]
which is the first equation of \eqref{eq:chiODE}. For the second, write $\eta=\chi\rho^{2}q(e^{i\theta})$. By \eqref{eq:sigma}, $d\rho^{2}/d\sigma=-2\rho^{2}$, and by the definition of $\varepsilon$, $dq(e^{i\theta})/d\sigma=\varepsilon\,q(e^{i\theta})$; the product rule therefore gives $\frac{d\eta}{d\sigma}=\rho^{2}q(e^{i\theta})\bigl(\tfrac{d\chi}{d\sigma}-2\chi+\varepsilon\chi\bigr)$, and comparison with the first equation yields $\frac{d\chi}{d\sigma}-2\chi+\varepsilon\chi=\kappa(1-\chi)$, i.e.\ the second. For $\varepsilon$, note that $\theta'=1-\zeta'=\alpha\sin\zeta/\rho$ by \eqref{eq:rhozeta} and $q(e^{i\theta})=N(\theta)/(a^{2}b^{2})$, so
\[
  \varepsilon=\frac{N'(\theta)\,\theta'/N(\theta)}{(\cos\Delta+n)\alpha/\rho}
   =\frac{N'(\theta)\,\sin\Delta}{N(\theta)\,(\cos\Delta+n)} ,
\]
and $|N'|\le a^{2}-b^{2}$, $N\ge b^{2}$, $\sin\Delta\le\Delta\le\delta_{*}$, $\cos\Delta+n\ge n$ give $|\varepsilon|\le C_{2}\delta_{*}/n\le\tfrac1{64}$ by \eqref{eq:Dstar}. As $\kappa<1$ by \eqref{eq:kappa}, it follows that $2-\kappa-\varepsilon>2-1-\tfrac1{64}=\tfrac{63}{64}\ge\tfrac12$.

Finally, put $\kappa_{0}:=n/(1+n)$, so that $-\kappa_{0}/(2-\kappa_{0})=-n/(2+n)$. Bringing the two fractions to a common denominator,
\[
  \chi_{\sharp}+\frac{n}{2+n}
  =-\frac{\kappa}{2-\kappa-\varepsilon}+\frac{\kappa_{0}}{2-\kappa_{0}}
  =-\frac{2(\kappa-\kappa_{0})+\kappa_{0}\varepsilon}
         {(2-\kappa-\varepsilon)(2-\kappa_{0})} .
\]
Since $n<1$ we have $\kappa_{0}<\tfrac12$, hence $2-\kappa_{0}>\tfrac32$, and together with $2-\kappa-\varepsilon>\tfrac{63}{64}$ the denominator exceeds $\tfrac{189}{128}>1$. Therefore $|\chi_{\sharp}+n/(2+n)|\le2|\kappa-\kappa_{0}|+\kappa_{0}|\varepsilon|$. Next,
\[
  |\kappa-\kappa_{0}|
  =\frac{n}{(\cos\Delta+n)(1+n)}\,(1-\cos\Delta)\le\frac{\Delta^{2}}{2},
\]
because the first factor is less than $1$ (as $\cos\Delta+n>n$ and $1+n>1$) and $1-\cos\Delta=2\sin^{2}(\Delta/2)\le\Delta^{2}/2$. With $\Delta\le\delta_{*}$, $\kappa_{0}<\tfrac12$ and $|\varepsilon|\le C_{2}\delta_{*}/n$ we obtain
\[
  \Bigl|\chi_{\sharp}+\frac{n}{2+n}\Bigr|
  \le\delta_{*}^{2}+\frac{C_{2}\delta_{*}}{2n}
  \le\frac1{64}+\frac1{128}<\frac18
\]
by \eqref{eq:Dstar}.
\end{proof}

\begin{proposition}[the capture configuration]\label{prop:capture}
Assume $a\ge b>0$ and $0<n<1$. Captured solutions exist. More precisely, for every sufficiently small $\rho_{0}>0$ there is an angle $\zeta_{0}^{*}(\rho_{0})\in(\pi-\delta_{*},\pi)$ such that the solution with $(\rho,\zeta)(\varphi_{0})=(\rho_{0},\zeta_{0}^{*}(\rho_{0}))$ remains in the corner $\{0\le\Delta\le\delta_{*}\}$ and is captured, and every angle in $(\pi-\delta_{*},\pi)$ with this property satisfies
\begin{equation}\label{eq:separatrix-asym}
  \zeta_{0}^{*}(\rho_{0})=\pi-\frac{f(\varphi_{0})}{2+n}\,\rho_{0}+o(\rho_{0})
  \qquad(\rho_{0}\downarrow0).
\end{equation}
The corresponding pursuer lies outside the ellipse, and quantitatively
\begin{equation}\label{eq:overshoot}
  q(\mathbf{P})-1=\frac{n}{2+n}\,q(e^{i\theta_{0}})\,\rho_{0}^{2}+o(\rho_{0}^{2}),
  \qquad \theta_{0}=\varphi_{0}-\zeta_{0}^{*} .
\end{equation}
\end{proposition}

\begin{proof}
Fix $\rho_{0}>0$ small and let $\zeta_{0}$ range over $(\pi-\delta_{*},\pi)$; write $\chi_{0}$ for the corresponding value of \eqref{eq:chidef} at $\varphi_{0}$. Everything below uses only Lemma~\ref{lem:corner}, that is the linear equation \eqref{eq:chiODE}, and \emph{not} the sign of $\eta$.

Set $\chi_{\flat}:=-n/(2+n)$. For $0<\gamma\le\tfrac14$ we call $\gamma$ \emph{admissible} for a solution if $|\chi_{\sharp}-\chi_{\flat}|\le\gamma/2$ at every point of the solution in the corner. By \eqref{eq:chisharp}, $\gamma=\tfrac14$ is admissible for every solution; moreover $\chi_{\flat}-\gamma\ge\chi_{\flat}-\tfrac14>-1$ because $n<1$. We claim that, for every admissible $\gamma$ and every time $\varphi_{1}\ge\varphi_{0}$ at which the solution lies in the corner (with $\sigma$ measured from $\varphi_{a}=\varphi_{1}$):
\begin{enumerate}
\item[\textup{(a)}] if $\chi(\varphi_{1})\ge\chi_{\flat}+\gamma$, the solution leaves the corner through $\Delta=\delta_{*}$;
\item[\textup{(b)}] if $\chi(\varphi_{1})\le\chi_{\flat}-\gamma$, the solution leaves the corner through $\Delta=0$.
\end{enumerate}

\emph{Proof of \textup{(a)}.} Suppose $\chi(\varphi_{1})\ge\chi_{\flat}+\gamma$. As long as the solution stays in the corner, $\chi\ge\chi_{\flat}+\gamma$ persists: at $\chi=\chi_{\flat}+\gamma$ one has $\chi-\chi_{\sharp}\ge\gamma-\gamma/2=\gamma/2$, so by \eqref{eq:chiODE} $d\chi/d\sigma=(2-\kappa-\varepsilon)(\chi-\chi_{\sharp})\ge\tfrac12\cdot\tfrac{\gamma}{2}>0$. Hence $d\chi/d\sigma\ge\tfrac12(\chi-\chi_{\flat}-\gamma/2)$ throughout, so $\chi-\chi_{\flat}-\gamma/2$ grows at least like $e^{\sigma/2}$ and $\chi\to+\infty$. Now $\sin\Delta=\tfrac12(1+\chi)\rho\,h\,q(e^{i\theta})$ by \eqref{eq:chiform}, and $h\,q(e^{i\theta})/f=N(\theta)/N(\varphi)$, so the $\Delta$-equation $\Delta'=\sin\Delta/s-1$ coming from \eqref{eq:rhozeta} reads
\[
  \Delta'=(1+\chi)\frac{N(\theta)}{2N(\varphi)}-1
  \ \ge\ (1+\chi)\frac{b^{2}}{2a^{2}}-1 .
\]
Once $1+\chi\ge4a^{2}/b^{2}$ this gives $\Delta'\ge1$, i.e.\ $\sin\Delta\ge2s$, so, using $\sin\Delta\ge2\Delta/\pi$ on $[0,\pi/2]$ and $\cos\Delta+n\le2$,
\[
  \frac{d\Delta}{d\sigma}=\frac{\Delta'}{d\sigma/d\varphi}
  =\frac{\sin\Delta-s}{\cos\Delta+n}
  \ \ge\ \frac{\sin\Delta}{2(\cos\Delta+n)}\ \ge\ \frac{\Delta}{2\pi},
\]
the first inequality since $\sin\Delta\ge2s$. Since $\sigma$ runs over the whole of $[0,\infty)$ if the solution stays in the corner, $\Delta$ would grow beyond every bound; hence $\Delta$ reaches $\delta_{*}$ at a positive value of $\rho$, and the solution leaves the corner with $\zeta$ decreasing away from $\pi$.

\emph{Proof of \textup{(b)}.} Suppose $\chi(\varphi_{1})\le\chi_{\flat}-\gamma$. Symmetrically $\chi\le\chi_{\flat}-\gamma$ persists and $d\chi/d\sigma\le\tfrac12(\chi-\chi_{\flat}+\gamma/2)\le-\tfrac{\gamma}{4}<0$, so $\chi$ decreases at least linearly in $\sigma$. But $\chi\ge-1$ as long as $\zeta\in(0,\pi]$, so $\chi$ must reach $-1$ at a finite $\sigma$, that is $\sin\zeta=0$ and $\Delta=0$: the solution leaves the corner on the other side, with $\zeta$ crossing $\pi$.

\emph{A captured solution.} Take $\gamma=\tfrac14$ and $\varphi_{1}=\varphi_{0}$. Leaving the corner through $\Delta=\delta_{*}$ and leaving it through $\Delta=0$ are open conditions on $\zeta_{0}$. Indeed, both exits from the corner are transversal: at $\Delta=\delta_{*}$ one has $d\Delta/d\sigma=(\sin\delta_{*}-s)/(\cos\delta_{*}+n)>0$, because $\rho$ decreases in the corner and hence $s\le f_{\max}\rho_{0}<\sin\delta_{*}$ once $\rho_{0}$ is small, and at $\Delta=0$ one has $\Delta'=-1$; so each exit persists under small changes of $\zeta_{0}$ by continuous dependence. The two outcomes are mutually exclusive, and each occurs: for $\zeta_{0}$ close to $\pi-\delta_{*}$ one has $1+\chi_{0}=2\sin\delta_{*}/(\rho_{0}h q(e^{i\theta_{0}}))$, which exceeds $\chi_{\flat}+\gamma+1$ once $\rho_{0}$ is small, so (a) applies, and for $\zeta_{0}$ close to $\pi$ one has $\sin\zeta_{0}\to0$, i.e.\ $\chi_{0}\to-1<\chi_{\flat}-\gamma$, so (b) applies. Since the interval $(\pi-\delta_{*},\pi)$ is connected, there is $\zeta_{0}^{*}$ realizing neither; the corresponding solution leaves the corner neither through $\Delta=\delta_{*}$ nor through $\Delta=0$, so it remains in the corner, where $d\rho/d\varphi\le-(\cos\delta_{*}+n)/f_{\max}<0$, so $\rho$ reaches $0$ in finite $\varphi$ and the solution is captured.

\emph{Asymptotics.} Let now $\zeta_{0}^{*}\in(\pi-\delta_{*},\pi)$ be any angle whose solution remains in the corner. Since this solution never leaves the corner, (a) and (b) with $\gamma=\tfrac14$, applied at every time $\varphi_{1}\ge\varphi_{0}$, give $|\chi-\chi_{\flat}|<\tfrac14$ along the whole solution. In particular $1+\chi$ is bounded, so $\sin\Delta=\tfrac12(1+\chi)\rho\,h\,q(e^{i\theta})\le c\rho$ by \eqref{eq:chiform}, with $c$ independent of $\rho_{0}$; since $\Delta\le\delta_{*}<\pi/2$ and $\rho\le\rho_{0}$, this gives $\bar\Delta:=\sup\Delta\le c'\rho_{0}$ along the solution. The proof of \eqref{eq:chisharp} uses $\Delta\le\delta_{*}$ only through $|\kappa-\kappa_{0}|\le\Delta^{2}/2$ and $|\varepsilon|\le C_{2}\Delta/n$; along this solution $\Delta\le\bar\Delta$, so it gives $|\chi_{\sharp}-\chi_{\flat}|\le\bar\Delta^{2}+C_{2}\bar\Delta/(2n)$. Hence $\bar\gamma:=2\bigl(\bar\Delta^{2}+C_{2}\bar\Delta/(2n)\bigr)=O(\rho_{0})$ is admissible for this solution as soon as $\rho_{0}$ is so small that $\bar\gamma\le\tfrac14$. Applying (a) and (b) with $\gamma=\bar\gamma$ at $\varphi_{1}=\varphi_{0}$ gives $|\chi_{0}-\chi_{\flat}|<\bar\gamma$, and $\bar\gamma\to0$ as $\rho_{0}\downarrow0$; that is,
\begin{equation}\label{eq:lam0asym}
  1+\chi(\varphi_{0})=\frac{2}{2+n}+o(1)\qquad(\rho_{0}\downarrow0).
\end{equation}
We now read off \eqref{eq:separatrix-asym} and \eqref{eq:overshoot} from \eqref{eq:lam0asym}. First, by \eqref{eq:chiform},
\[
  \sin(\pi-\zeta_{0}^{*})=\sin\zeta_{0}^{*}
  =\tfrac12\bigl(1+\chi(\varphi_{0})\bigr)\rho_{0}\,h(\varphi_{0})\,q(e^{i\theta_{0}})
  =\frac{\rho_{0}\,h(\varphi_{0})\,q(e^{i\theta_{0}})}{2+n}+o(\rho_{0}),
\]
where $h(\varphi_{0})q(e^{i\theta_{0}})\to h(\varphi_{0})N(\varphi_{0})/(a^{2}b^{2}) =f(\varphi_{0})$ because $N$ is $\pi$-periodic and $\theta_{0}=\varphi_{0}-\zeta_{0}^{*}\to\varphi_{0}-\pi$; since $\pi-\zeta_{0}^{*}\to0$, this is \eqref{eq:separatrix-asym}. Second, by \eqref{eq:eta} and \eqref{eq:chidef}, $q(\mathbf{P})-1=-\eta=-\chi(\varphi_{0})\,q(e^{i\theta_{0}})\rho_{0}^{2}$, and \eqref{eq:lam0asym} gives \eqref{eq:overshoot}.
\end{proof}

\subsection{Periodic Solution for $n<1$}\label{sec:periodic}

This subsection shows that the solutions avoiding capture converge to a unique $\pi$-periodic solution. The proof is organized around the half-period map $\mathcal{Q}$ of \eqref{eq:halfperiod} on $K$: it is well defined by Corollary~\ref{cor:Kgeom} and Proposition~\ref{prop:nocapture}, contractive by Lemma~\ref{lem:contraction}, and has a unique attracting fixed point by Edelstein's fixed-point theorem \cite{Edelstein}; Lemma~\ref{lem:limitcurve} controls the one point, $\mathbf{E}(\varphi_{0})$, at which $\mathcal{Q}$ is defined only by continuity.

\subsubsection{A contraction estimate}

\begin{lemma}\label{lem:contraction}
Let $z_{1},z_{2}$ be two solutions of \eqref{equation-rev-2}, on any interval on which both exist with $|z_{1}|,|z_{2}|>0$, and set $L(\varphi):=|z_{1}(\varphi)-z_{2}(\varphi)|^{2}$.
\begin{enumerate}
\item[\textup{(a)}] $\dfrac{dL}{d\varphi} =-\dfrac{2n}{f(\varphi)}(\rho_{1}+\rho_{2}) \bigl(1-\cos(\zeta_{1}-\zeta_{2})\bigr)\le0$;
\item[\textup{(b)}] if $\varphi_{a}<\varphi_{b}$, both solutions exist with $|z_{1}|,|z_{2}|>0$ on $[\varphi_{a},\varphi_{b}]$ and $z_{1}(\varphi_{a})\ne z_{2}(\varphi_{a})$, then $L(\varphi_{b})<L(\varphi_{a})$.
\end{enumerate}
\end{lemma}

\begin{proof}
(a) From \eqref{equation-rev-2},
\begin{align}
    \frac{dL}{d\varphi}
    &= \left(
        \frac{d\overline{z_1}}{d\varphi}-\frac{d\overline{z_2}}{d\varphi}
      \right)(z_1-z_2)+\left(
        \frac{dz_1}{d\varphi}-\frac{dz_2}{d\varphi}
      \right)\overline{(z_1-z_2)}\notag\\
    &=-\frac{n}{f(\varphi)}\left\{\left(\frac{\overline{z_1}}{|z_1|}-\frac{\overline{z_2}}{|z_2|}\right)(z_1-z_2)+\left(\frac{z_1}{|z_1|}-\frac{z_2}{|z_2|}\right)\overline{(z_1-z_2)}\right\}\notag\\
    &=-\frac{2n}{f(\varphi)}\left\{(|z_1|+|z_2|)-\Re[z_1\overline{z_2}]\left(\frac{1}{|z_1|}+\frac{1}{|z_2|}\right)\right\}.
\end{align}
Here the term $1/f(\varphi)$ cancels in the difference, and the term $iz$ contributes $2\Re\bigl[i|z_1-z_2|^{2}\bigr]=0$. Substituting $z_j=\rho_j e^{i\zeta_j}$, $j=1,2$, and $\Re[z_1\overline{z_2}]=\rho_1\rho_2\cos(\zeta_1-\zeta_2)$, the bracket collapses to $(\rho_1+\rho_2)\bigl(1-\cos(\zeta_1-\zeta_2)\bigr)$, which yields
\begin{align}
    \frac{dL}{d\varphi}
    = -\frac{2n}{f(\varphi)}(\rho_1+\rho_2)\bigl(1-\cos(\zeta_1-\zeta_2)\bigr)
    \leq 0.
    \label{dLdphi}
\end{align}

(b) Suppose not; then $L$ is constant on $[\varphi_{a},\varphi_{b}]$, so $dL/d\varphi\equiv0$ and, by \eqref{dLdphi}, $\cos(\zeta_1-\zeta_2)\equiv1$: the two solutions share their argument, so $\sin\zeta_1\equiv\sin\zeta_2$. The right-hand side of \eqref{equation-rev-2} being smooth on $\{z\ne0\}$, where both solutions remain, uniqueness in either direction of $\varphi$ forbids them to agree anywhere, since they differ at $\varphi_a$; hence $\mu_1\ne\mu_2$ throughout and, subtracting the two $\zeta$-equations in \eqref{equation-rev},
\[
0=\frac{d(\zeta_1-\zeta_2)}{d\varphi}
=-\frac{\sin\zeta_1}{f}\bigl(e^{-\mu_1}-e^{-\mu_2}\bigr) ,
\]
so $\sin\zeta_1\equiv0$. Then $\zeta_1$, continuous with values in $\pi\mathbb{Z}$, is constant and $d\zeta_1/d\varphi\equiv0$, whereas \eqref{rev-zeta} gives $d\zeta_1/d\varphi\equiv1$ --- a contradiction. Since $L$ is non-increasing by (a), the strict inequality
\begin{align}
L(\varphi_b) < L(\varphi_a)
\label{strict}
\end{align}
follows.
\end{proof}

Lemma~\ref{lem:contraction} provides no \emph{rate}: the right-hand side of (a) vanishes whenever $\zeta_{1}=\zeta_{2}$, so a purely radial separation of two solutions is not contracted at that instant. Accordingly the half-period map $\mathcal{Q}$ of \eqref{eq:halfperiod} below is shown to be contractive only in the sense of Edelstein, and it is the compactness of $K$ that converts strict decrease into convergence.

\subsubsection{Limit curves}

The half-period map used in the next subsubsection is defined by \eqref{eq:veq} at every point of $K$ except one, namely $\mathbf{E}(\varphi_{0})$, where it can only be defined by continuity. The following lemma is what keeps that one point under control: a uniform limit of solutions inside $K$ never meets the evader after the initial instant.

\begin{lemma}[limit curves never touch the evader]\label{lem:limitcurve}
Let $I=[\varphi_{a},\varphi_{b}]$ with $\varphi_{a}<\varphi_{b}$, let $\mathbf{P}_{j}:I\to K$ be solutions of \eqref{eq:veq} with $\rho_{j}>0$, and suppose $\mathbf{P}_{j}\to\mathbf{P}_{\infty}$ uniformly on $I$. Then $\mathbf{P}_{\infty}(\varphi)\ne\mathbf{E}(\varphi)$ for every $\varphi\in(\varphi_{a},\varphi_{b}]$.
\end{lemma}

\begin{proof}
$\mathbf{P}_{\infty}$ takes values in $K$ and satisfies $|\mathbf{P}_{\infty}'|\le n\alpha$ a.e., being a uniform limit of curves with this property. Let $Z:=\{\varphi\in I:\mathbf{P}_{\infty}(\varphi)=\mathbf{E}(\varphi)\}$; it is closed. It has empty interior, for on an interval contained in $Z$ we would have $\alpha=|\mathbf{E}'|=|\mathbf{P}_{\infty}'|\le n\alpha$ a.e., i.e. $n\ge1$.

On any interval $J\subseteq I$ on which $|\mathbf{E}-\mathbf{P}_{\infty}|$ is bounded away from $0$, the right-hand side of \eqref{eq:veq} converges uniformly along $\mathbf{P}_{j}$, and passing to the limit in the integral form of \eqref{eq:veq} shows that $\mathbf{P}_{\infty}$ solves \eqref{eq:veq} on $J$.

Suppose now some $\varphi_{c}\in Z$ satisfies $\varphi_{c}>\varphi_{a}$. As $Z$ has empty interior, there is $\tau\in(\varphi_{a},\varphi_{c})\setminus Z$; put $\beta:=\inf\bigl(Z\cap(\tau,\infty)\bigr)\le\varphi_{c}$, so that $\beta>\tau$. On every compact subinterval of $(\tau,\beta)$ the quantity $|\mathbf{E}-\mathbf{P}_{\infty}|$ is bounded away from $0$, so by the previous paragraph $\mathbf{P}_{\infty}$ solves \eqref{eq:veq} on $(\tau,\beta)$, with values in $K$ and with $|\mathbf{E}-\mathbf{P}_{\infty}|\to0$ as $\varphi\uparrow\beta$. This contradicts Proposition~\ref{prop:nocapture}, applied with the initial time $\varphi_{0}$ replaced by $\tau$ (its proof does not use the choice of the initial time).
\end{proof}

\subsubsection{Existence, uniqueness and attractivity}

\begin{theorem}[Existence, uniqueness and attractivity]
Assume $a\ge b>0$ and $0<n<1$. Then:
\begin{enumerate}
\item[\textup{(i)}] equation \eqref{equation-rev-2} has a unique $\pi$-periodic solution $z^{*}$; it is persistently evading, its pursuer satisfies $q(\mathbf{P}^{*})<1$, and $0<r^{*}\le|z^{*}(\varphi)|\le2a$ with $r^{*}:=\min_{\varphi}|z^{*}(\varphi)|$;
\item[\textup{(ii)}] if the pursuer lies in the closed elliptical disc $K$ at some time and does not coincide with the evader, then $|z(\varphi)-z^{*}(\varphi)|\to0$ as $\varphi\to\infty$; for a fixed initial time $\varphi_{0}$ the convergence is uniform with respect to the initial position $\mathbf{p}\in K_{0}$ of the pursuer;
\item[\textup{(iii)}] more generally, every persistently evading solution satisfies $|z(\varphi)-z^{*}(\varphi)|\to0$.
\end{enumerate}
\label{thm:stability}
\end{theorem}

\begin{proof}
Throughout, $\mathbf{P}_{\mathbf{p}}$ denotes the solution of \eqref{eq:veq} with $\mathbf{P}_{\mathbf{p}}(\varphi_{0})=\mathbf{p}$, and we use the dictionary \eqref{eq:dict}--\eqref{eq:isom} to pass between the two pictures without further comment.

\textbf{Step 1: the half-period map.} By Proposition~\ref{prop:nocapture}, for every $\mathbf{p}\in K_{0}$ the solution $\mathbf{P}_{\mathbf{p}}$ exists on $[\varphi_{0},\infty)$ with $\rho>0$ and takes values in $K$. Since $K=-K$ we may define
\begin{equation}\label{eq:halfperiod}
  \mathcal{Q}:K_{0}\to K,\qquad
  \mathcal{Q}(\mathbf{p}):=-\mathbf{P}_{\mathbf{p}}(\varphi_{0}+\pi).
\end{equation}
Because $\alpha(\varphi+\pi)=\alpha(\varphi)$ and $\mathbf{E}(\varphi+\pi)=-\mathbf{E}(\varphi)$, the curve $\varphi\mapsto-\mathbf{P}_{\mathbf{p}}(\varphi+\pi)$ is again a solution of \eqref{eq:veq}; by induction,
\begin{equation}\label{eq:Qiterate}
  \mathcal{Q}^{k}(\mathbf{p})=(-1)^{k}\mathbf{P}_{\mathbf{p}}(\varphi_{0}+k\pi),
  \qquad k\in\mathbb{N},
\end{equation}
and every iterate is defined, because $\mathbf{P}_{\mathbf{p}}(\varphi_{0}+k\pi)\ne\mathbf{E}(\varphi_{0}+k\pi) =(-1)^{k}\mathbf{E}(\varphi_{0})$ by $\rho>0$, so $\mathcal{Q}^{k}(\mathbf{p})\in K_{0}$. By \eqref{eq:isom} and Lemma~\ref{lem:contraction}, $\mathcal{Q}$ is $1$-Lipschitz on $K_{0}$ and moreover
\begin{equation}\label{eq:contractive}
  \bigl|\mathcal{Q}(\mathbf{p}_{1})-\mathcal{Q}(\mathbf{p}_{2})\bigr|
  <|\mathbf{p}_{1}-\mathbf{p}_{2}|
  \qquad\bigl(\mathbf{p}_{1},\mathbf{p}_{2}\in K_{0},\ \mathbf{p}_{1}\ne\mathbf{p}_{2}\bigr),
\end{equation}
that is, $\mathcal{Q}$ is \emph{contractive} in the sense of Edelstein. Since $K_{0}$ is dense in the compact set $K$ and $\mathbb{C}$ is complete, $\mathcal{Q}$ extends uniquely to a $1$-Lipschitz map $\widetilde{\mathcal{Q}}:K\to K$; no contractivity is claimed for the extension at the added point, and none will be needed.

\textbf{Step 2: a compact invariant subset of $K_{0}$.} Write $\mathbf{e}_{0}:=\mathbf{E}(\varphi_{0})$ and put
\begin{equation}\label{eq:Adef}
  \mathcal{A}:=\cl\bigl(\mathcal{Q}(K_{0})\bigr)\ \subseteq\ K ,
\end{equation}
where $\cl$ denotes the closure in $\mathbb{C}$. By Corollary~\ref{cor:Kgeom}(b) and $K=-K$, the pursuer of any $\mathbf{p}\in K_{0}$ satisfies $\mathbf{P}_{\mathbf{p}}(\varphi_{0}+\pi)\in\operatorname{int}K$, so
\begin{equation}\label{eq:QintK}
  \mathcal{Q}(K_{0})\subseteq\operatorname{int}K\subseteq K_{0} ,
\end{equation}
whence $\mathcal{A}$ is compact and, $\widetilde{\mathcal{Q}}$ being continuous, $\widetilde{\mathcal{Q}}(\mathcal{A})\subseteq \mathcal{A}$.

We claim that $\mathbf{e}_{0}\notin \mathcal{A}$, so that $\mathcal{A}\subseteq K_{0}$. Suppose $\mathcal{Q}(\mathbf{p}_{j})\to\mathbf{e}_{0}$ with $\mathbf{p}_{j}\in K_{0}$; passing to a subsequence, $\mathbf{p}_{j}\to\mathbf{p}\in K$, and by continuity $\widetilde{\mathcal{Q}}(\mathbf{p})=\mathbf{e}_{0}$. If $\mathbf{p}\in K_{0}$ this is impossible, because $\mathcal{Q}(\mathbf{p})\in\operatorname{int}K$ by \eqref{eq:QintK} whereas $q(\mathbf{e}_{0})=1$. There remains $\mathbf{p}=\mathbf{e}_{0}$, i.e.\ $\mathbf{p}_{j}\to\mathbf{e}_{0}$. By \eqref{eq:isom} and Lemma~\ref{lem:contraction}(a), $|\mathbf{P}_{\mathbf{p}_{j}}(\varphi)-\mathbf{P}_{\mathbf{p}_{k}}(\varphi)|\le|\mathbf{p}_{j}-\mathbf{p}_{k}|$ on $[\varphi_{0},\varphi_{0}+\pi]$, so $(\mathbf{P}_{\mathbf{p}_{j}})$ is uniformly Cauchy there and converges uniformly to some curve $\mathbf{P}_{\infty}$ with $\widetilde{\mathcal{Q}}(\mathbf{e}_{0})=\lim_{j}\mathcal{Q}(\mathbf{p}_{j})=-\mathbf{P}_{\infty}(\varphi_{0}+\pi)$. By Lemma~\ref{lem:limitcurve}, $\mathbf{P}_{\infty}(\varphi_{0}+\pi)\ne\mathbf{E}(\varphi_{0}+\pi)=-\mathbf{e}_{0}$, that is $\widetilde{\mathcal{Q}}(\mathbf{e}_{0})\ne\mathbf{e}_{0}$ --- again impossible. Hence $\mathcal{A}\subseteq K_{0}$ and, by \eqref{eq:contractive}, $\mathcal{Q}$ restricts to a contractive self-map of the non-empty compact metric space $\mathcal{A}$.

\textbf{Step 3: existence and uniqueness, i.e.\ (i).} By Edelstein's fixed-point theorem \cite{Edelstein} applied to $\mathcal{Q}:\mathcal{A}\to \mathcal{A}$, there is a unique $\mathbf{p}^{*}\in \mathcal{A}$ with $\mathcal{Q}(\mathbf{p}^{*})=\mathbf{p}^{*}$, and $\mathcal{Q}^{k}(\mathbf{p})\to\mathbf{p}^{*}$ for every $\mathbf{p}\in \mathcal{A}$. Both conclusions extend from $\mathcal{A}$ to $K_{0}$: by \eqref{eq:Adef} every $\mathbf{p}\in K_{0}$ has $\mathcal{Q}(\mathbf{p})\in \mathcal{A}$, so $\mathcal{Q}^{k}(\mathbf{p})=\mathcal{Q}^{k-1}(\mathcal{Q}(\mathbf{p}))\to\mathbf{p}^{*}$; and a fixed point of $\mathcal{Q}$ in $K_{0}$ lies in $\mathcal{Q}(K_{0})\subseteq \mathcal{A}$, so $\mathbf{p}^{*}$ is the only one in $K_{0}$.

Set $\mathbf{P}^{*}:=\mathbf{P}_{\mathbf{p}^{*}}$, so that $\mathbf{P}^{*}(\varphi_{0}+\pi)=-\mathbf{p}^{*}=-\mathbf{P}^{*}(\varphi_{0})$. Since $\varphi\mapsto-\mathbf{P}^{*}(\varphi+\pi)$ is a solution of \eqref{eq:veq} taking the value $\mathbf{p}^{*}$ at $\varphi_{0}$, uniqueness of solutions gives $\mathbf{P}^{*}(\varphi+\pi)=-\mathbf{P}^{*}(\varphi)$ for every $\varphi\ge\varphi_{0}$ (and this identity extends $\mathbf{P}^{*}$ to all $\varphi\in\mathbb{R}$), i.e.\ $z^{*}$ is $\pi$-periodic by \eqref{eq:isom}. Moreover $\rho^{*}=|\mathbf{E}-\mathbf{P}^{*}|\le\operatorname{diam}K=2a$ because both points lie in $K$; $q(\mathbf{P}^{*})<1$ by Corollary~\ref{cor:Kgeom}(c); and $|z^{*}|$ is continuous, positive and $\pi$-periodic, so $r^{*}=\min_{\varphi}|z^{*}(\varphi)|>0$; in particular $z^{*}$ is persistently evading.

Conversely, let $z^{**}$ be any $\pi$-periodic solution of \eqref{equation-rev-2}. Its pursuer lies in $K$ by Corollary~\ref{cor:Kgeom}(c), and $\mathbf{P}^{**}(\varphi_{0})\ne\mathbf{E}(\varphi_{0})$ because $\rho>0$, so $\mathbf{P}^{**}(\varphi_{0})\in K_{0}$ and, by \eqref{eq:isom}, it is a fixed point of $\mathcal{Q}$ lying in $K_{0}$; by Step~3 it equals $\mathbf{p}^{*}$, so $z^{**}=z^{*}$. This proves (i).

\textbf{Step 4: uniform attractivity on $K$, i.e.\ (ii).} Put $g_{k}(\mathbf{p}):=\bigl|\widetilde{\mathcal{Q}}^{\,k}(\mathbf{p})-\mathbf{p}^{*}\bigr|$ for $\mathbf{p}\in K$. Each $g_{k}$ is continuous, and $g_{k+1}\le g_{k}$ because $\widetilde{\mathcal{Q}}$ is $1$-Lipschitz and fixes $\mathbf{p}^{*}$. On $K_{0}$ we have $g_{k}\to0$ by Step~3, and at the one remaining point $\widetilde{\mathcal{Q}}(\mathbf{e}_{0})\in \mathcal{A}$ by \eqref{eq:Adef}, so $\widetilde{\mathcal{Q}}^{\,k}(\mathbf{e}_{0}) =\mathcal{Q}^{k-1}\bigl(\widetilde{\mathcal{Q}}(\mathbf{e}_{0})\bigr)\to\mathbf{p}^{*}$ as well; thus $g_{k}\to0$ pointwise on all of $K$. By Dini's theorem the convergence is uniform. Now let $\mathbf{p}\in K_{0}$. By \eqref{eq:Qiterate}, $|\mathbf{P}_{\mathbf{p}}(\varphi_{0}+k\pi)-\mathbf{P}^{*}(\varphi_{0}+k\pi)| =|\mathcal{Q}^{k}(\mathbf{p})-\mathcal{Q}^{k}(\mathbf{p}^{*})|=g_{k}(\mathbf{p})$, so, since $|\mathbf{P}-\mathbf{P}^{*}|$ is non-increasing in $\varphi$, for $\varphi\in[\varphi_{0}+k\pi,\varphi_{0}+(k+1)\pi]$,
\[
  |z(\varphi)-z^{*}(\varphi)|=|\mathbf{P}(\varphi)-\mathbf{P}^{*}(\varphi)|
  \le|\mathbf{P}(\varphi_{0}+k\pi)-\mathbf{P}^{*}(\varphi_{0}+k\pi)|
  =g_{k}(\mathbf{p})\le\sup_{K}g_{k}\longrightarrow0 ,
\]
uniformly over $\mathbf{p}\in K_{0}$. The same argument applies verbatim with $\varphi_{0}$ replaced by any later time at which the pursuer is known to lie in $K$, which proves (ii). In particular such a solution is persistently evading, since $\liminf_{\varphi\to\infty}|z|\ge r^{*}>0$.

\textbf{Step 5: persistently evading solutions, i.e.\ (iii).} Let $\hat z$ be persistently evading, with pursuer $\hat{\mathbf{P}}$, and put $L(\varphi):=|\hat z(\varphi)-z^{*}(\varphi)|^{2}$, which is non-increasing by Lemma~\ref{lem:contraction}(a); let $L\downarrow L_{\infty}\ge0$. Choose $\varphi_{1}\ge\varphi_{0}$ with $m:=\inf_{\varphi\ge\varphi_{1}}|\hat z(\varphi)|>0$. By Corollary~\ref{cor:Kgeom}(a), $q(\hat{\mathbf{P}}(\varphi))\le R_{1}:=\max\{q(\hat{\mathbf{P}}(\varphi_{1})),1\}$ for $\varphi\ge\varphi_{1}$, whence $|\hat{\mathbf{P}}|\le a\sqrt{R_{1}}$ and therefore $|\hat z|=|\mathbf{E}-\hat{\mathbf{P}}|\le a(1+\sqrt{R_{1}})=:C_{3}$ there. Hence the points $w_{k}:=\hat z(\varphi_{1}+k\pi)$ lie in the compact set $\{m\le|w|\le C_{3}\}$, and a subsequence converges, $w_{k_{j}}\to w^{\star}$ with $|w^{\star}|\ge m$.

Write $z(\,\cdot\,;w)$ for the solution of \eqref{equation-rev-2} with $z(\varphi_{1};w)=w$; since $f$ is $\pi$-periodic, $z(\varphi;w_{k})=\hat z(\varphi+k\pi)$. We claim that $z(\,\cdot\,;w^{\star})$ survives $[\varphi_{1},\varphi_{1}+\pi]$ and that $z(\varphi_{1}+\pi;w_{k_{j}})\to z(\varphi_{1}+\pi;w^{\star})$. Indeed, take $j$ with $|w^{\star}-w_{k_{j}}|<m/2$ and let $J$ be the set of $\varphi\in[\varphi_{1},\varphi_{1}+\pi]$ up to which $z(\,\cdot\,;w^{\star})$ exists without vanishing. On $J$ both $z(\,\cdot\,;w^{\star})$ and $z(\,\cdot\,;w_{k_{j}})$ are non-vanishing, so Lemma~\ref{lem:contraction}(a) gives $|z(\varphi;w^{\star})-z(\varphi;w_{k_{j}})|\le|w^{\star}-w_{k_{j}}|<m/2$ and hence $|z(\varphi;w^{\star})|>m/2$; the solution therefore cannot approach $0$ at the right endpoint of $J$ and may be continued, so $J=[\varphi_{1},\varphi_{1}+\pi]$ and the same estimate holds at $\varphi_{1}+\pi$. Note that this uses only that the reference solution stays above $m$, not that its pursuer lies in $K$.

Since $z^{*}$ is $\pi$-periodic, $|w_{k}-z^{*}(\varphi_{1})|^{2}=L(\varphi_{1}+k\pi)$ and $|z(\varphi_{1}+\pi;w_{k})-z^{*}(\varphi_{1})|^{2}=L(\varphi_{1}+(k+1)\pi)$; letting $j\to\infty$ gives
\[
  |w^{\star}-z^{*}(\varphi_{1})|^{2}=L_{\infty}
  =|z(\varphi_{1}+\pi;w^{\star})-z^{*}(\varphi_{1})|^{2},
\]
so the pair $\bigl(z(\,\cdot\,;w^{\star}),z^{*}\bigr)$ does not decrease $L$ over $[\varphi_{1},\varphi_{1}+\pi]$. If $L_{\infty}>0$ then $w^{\star}\ne z^{*}(\varphi_{1})$, contradicting \eqref{strict}. Hence $L_{\infty}=0$, and since $L$ is non-increasing, $|\hat z(\varphi)-z^{*}(\varphi)|\to0$. This proves (iii).
\end{proof}

\subsubsection{Exponential stability and the Floquet multipliers}

Theorem~\ref{thm:stability} asserts convergence but not a rate. Linearizing along $z^{*}$ supplies one, and shows in addition that $z^{*}$ is hyperbolic.

\begin{proposition}[exponential stability]\label{prop:floquet}
Assume $a\ge b>0$ and $0<n<1$, and write $z^{*}=e^{\mu^{*}+i\zeta^{*}}$, $\rho^{*}:=e^{\mu^{*}}$, for the $\pi$-periodic solution of Theorem~\ref{thm:stability}. Then its Floquet multipliers $m_{1},m_{2}$ satisfy
\begin{equation}\label{eq:floquet}
  m_{1}m_{2}=e^{-nI},\qquad
  I:=\int_{\varphi_{0}}^{\varphi_{0}+\pi}
        \frac{\alpha(\varphi)}{\rho^{*}(\varphi)}\,d\varphi
  \ \in\ \Bigl[\frac{\pi\alpha_{\min}}{2a},\ \frac{\pi\alpha_{\max}}{r^{*}}\Bigr],
\end{equation}
and $|m_{1}|,|m_{2}|<1$. In particular $z^{*}$ is hyperbolic and exponentially asymptotically stable.
\end{proposition}

\begin{proof}
Write the variational equation of \eqref{equation-rev} along $z^{*}$ in the coordinates $(u,v)=(\delta\mu,\delta\zeta)$:
\begin{equation}\label{eq:varioeq}
  u'=-Au-Bv,\qquad v'=Bu-Cv,
\end{equation}
where
\[
  A:=\frac{\alpha(\cos\zeta^{*}-n)}{\rho^{*}},\qquad
  B:=\frac{\alpha\sin\zeta^{*}}{\rho^{*}},\qquad
  C:=\frac{\alpha\cos\zeta^{*}}{\rho^{*}} ,
\]
so that $A=d\mu^{*}/d\varphi$ by \eqref{rev-mu} and $C-A=n\alpha/\rho^{*}$. Let $M$ be the monodromy matrix of \eqref{eq:varioeq} over $[\varphi_{0},\varphi_{0}+\pi]$; since $z\mapsto(\mu,\zeta)$ is a diffeomorphism away from $z=0$, the linearized Poincar\'e map of \eqref{equation-rev-2} at $z^{*}(\varphi_{0})$ is conjugate to $M$, and $m_{1},m_{2}$ are its eigenvalues.

\emph{(1) The determinant.} By Liouville's formula, $\det M=\exp\int(-A-C)\,d\varphi$, the integral being over one period. Since $A=d\mu^{*}/d\varphi$ and $\mu^{*}$ is $\pi$-periodic, $\int A\,d\varphi=0$; equivalently, integrating \eqref{rev-mu} over one period gives $\int\alpha\rho^{*-1}\cos\zeta^{*}\,d\varphi=n\int\alpha\rho^{*-1}\,d\varphi$, that is $\int C\,d\varphi=nI$. Hence $\det M=e^{-nI}$, which is the first part of \eqref{eq:floquet}; the bounds on $I$ follow from \eqref{eq:alpha} together with $r^{*}\le\rho^{*}\le2a$, and in particular $0<\det M<1$.

\emph{(2) A linearized Lyapunov function.} Put $\Lambda:=e^{2\mu^{*}}(u^{2}+v^{2})$. This is the linearization along $z^{*}$ of the function $L$ of Lemma~\ref{lem:contraction}, because $\delta z=e^{\mu^{*}+i\zeta^{*}}(u+iv)$ gives $|\delta z|^{2}=\Lambda$. Along \eqref{eq:varioeq} the coupling terms cancel, $\frac{d}{d\varphi}(u^{2}+v^{2})=-2Au^{2}-2Cv^{2}$, and therefore
\begin{equation}\label{eq:Lambdadecay}
  \frac{d\Lambda}{d\varphi}=2e^{2\mu^{*}}(A-C)\,v^{2}
  =-2n\alpha\rho^{*}v^{2}\ \le\ 0 ,
\end{equation}
which is exactly \eqref{dLdphi} linearized. Since $\mu^{*}$ is $\pi$-periodic, the weight $e^{2\mu^{*}}$ takes the same value at the two endpoints, so $|Mx|\le|x|$ for every $x\in\mathbb{R}^{2}$; thus $\|M\|\le1$ and $|m_{i}|\le1$.

\emph{(3) Strictness.} Suppose $|m_{i}|=1$ for some $i$. If $m_{1},m_{2}$ were a non-real conjugate pair of modulus one, then $\det M=1$, contradicting (1). So some eigenvalue is $m=\pm1$, with an eigenvector $x\ne0$; then $|Mx|=|x|$, so $\Lambda$ is constant along the corresponding solution of \eqref{eq:varioeq}, and \eqref{eq:Lambdadecay} forces $v\equiv0$. Then $0=v'=Bu$. But $B>0$ everywhere: by Theorem~\ref{thm:stability}(i) the pursuer attached to $z^{*}$ satisfies $q(\mathbf{P}^{*})<1$, i.e.\ $\eta>0$, so the first identity of \eqref{eq:etaform} gives
\begin{equation}\label{eq:sinzetapos}
  \sin\zeta^{*}\ \ge\ \tfrac12\,h(\varphi)\,\rho^{*}\,q(e^{i\theta})\ >\ 0 .
\end{equation}
Hence $u\equiv0$, so $x=0$, a contradiction. Therefore $|m_{1}|,|m_{2}|<1$, and $z^{*}$ is hyperbolic and exponentially asymptotically stable by Floquet theory \cite{Wiggins}.
\end{proof}

\begin{remark}\label{rem:floquet}
(a) If $m_{1},m_{2}$ are non-real, they form a conjugate pair, so $|m_{1}|=|m_{2}|=e^{-nI/2}$ by \eqref{eq:floquet}.

(b) Since $\alpha\,d\varphi=dt$ by \eqref{phi2t} and $\rho^{*}=n\lambda$ by \eqref{eq:lambdarho}, \eqref{eq:floquet} admits the kinematic reading
\begin{equation}\label{eq:floquet-kin}
  m_{1}m_{2}=\exp\Bigl(-\int\frac{dt}{\lambda(t)}\Bigr),
\end{equation}
the integral being taken over the half revolution of the evader that corresponds to $\varphi\in[\varphi_{0},\varphi_{0}+\pi]$: the linearized Poincar\'e map contracts area by exactly the accumulated reciprocal time-to-intercept over half a period.

(c) In the circular case $a=b$ the system \eqref{equation-rev} is autonomous, with $\alpha\equiv a$, and its equilibrium is given by $\cos\zeta^{*}=n$ and $\rho^{*}=\alpha\sin\zeta^{*}=a\sqrt{1-n^{2}}$, so that $\alpha/\rho^{*}\equiv1/\sqrt{1-n^{2}}$ and $I=\pi/\sqrt{1-n^{2}}$. The linearization at this equilibrium is computed explicitly in \cite[Sect.~3.2]{JSIAML}: in the variable $t$ its eigenvalues are $\omega_{\pm}=\bigl(-n\pm\sqrt{5n^{2}-4}\bigr)/\bigl(2a\sqrt{1-n^{2}}\bigr)$. Since $d\varphi/dt=1/a$ when $a=b$, the trace in the variable $\varphi$ is $a(\omega_{+}+\omega_{-})=-n/\sqrt{1-n^{2}}$, so that $\det M=\exp\bigl(-n\pi/\sqrt{1-n^{2}}\bigr)=e^{-nI}$, in agreement with \eqref{eq:floquet}. The eigenvalues are non-real exactly when $n<2/\sqrt5$: the equilibrium is then a focus and $m_{1},m_{2}$ form a conjugate pair, whereas for $2/\sqrt5\le n<1$ it is a node and $m_{1},m_{2}=e^{\pi a\omega_{\pm}}$ are real.
\end{remark}

\section{Verification of Simulation Results}\label{sec:verification}
This section verifies that the theorems proved in Section~\ref{sec:analysis} are consistent with the numerical simulations presented in \cite[Sects.~3 and~4]{JSIAML}. Simulation results show that the simulations of \cite{JSIAML} are in quantitative agreement with the finite-time blow-up bound for $n>1$, with the existence of a unique $\pi$-periodic solution attracting every solution that avoids capture for $0<n<1$, with the exponential rate of that convergence, and with the asymptotic location of the exceptional captured data, all established in Section~\ref{sec:analysis}. All simulations were performed with Mathematica Online~\cite{Mathematica}; the Mathematica code was generated with the assistance of a large language model (see Statements and Declarations) and was checked by the author against the analytical results of Section~\ref{sec:analysis}.

Except in Section~\ref{sec:num-capture} and for the four initial conditions of Section~\ref{sec:num-uniform}, the initial conditions are $\varphi_0=\pi/2$, $\rho(\varphi_0)=1$ and $\zeta(\varphi_0)=\pi/2$, corresponding to the evader starting at $(a,0)$ and the pursuer at the origin. In particular $\mu(\varphi_0)=\log\rho(\varphi_0)=0$, so $e^{\mu(\varphi_0)}=1$.

\subsection{Finite-time capture for $n>1$}\label{sec:num-blowup}
We integrated \eqref{equation-rev} numerically from the initial conditions above and recorded the first value $\varphi_B$ at which $\rho=e^{\mu}<10^{-8}$; the circular and elliptical examples are those of \cite[Figs.~2 and~5]{JSIAML}, and the values agree with the capture times reported there \cite[Sects.~3.1 and~4.1]{JSIAML}. Table~\ref{tab:blowup} compares $\varphi_B-\varphi_0$ with the upper bound $a\,e^{\mu(\varphi_0)}/(b^{2}(n-1))$ of Theorem~\ref{thm:Finite-time}.

\begin{table}[h]
\caption{Blow-up time $\varphi_B-\varphi_0$ for $n>1$ and the upper bound of
Theorem~\ref{thm:Finite-time}, with $\varphi_0=\pi/2$, $\rho(\varphi_0)=1$ and
$\zeta(\varphi_0)=\pi/2$}\label{tab:blowup}
\begin{tabular*}{\textwidth}{@{\extracolsep\fill}lcc}
\toprule
 & Circle & Ellipse \\
\midrule
$a$ & $1$ & $1$ \\
$b$ & $1$ & $0.5$ \\
$n$ & $1.2$ & $1.2$ \\
$\varphi_B-\varphi_0$ (numerical) & $1.674$ & $1.432$ \\
Upper bound $a\,e^{\mu(\varphi_0)}/(b^{2}(n-1))$ & $5$ & $20$ \\
\bottomrule
\end{tabular*}
\end{table}

In both cases the bound holds but is far from sharp: it exceeds the numerical value by a factor of about $3$ for the circle and about $14$ for the ellipse. The reason is that the proof of Theorem~\ref{thm:Finite-time} estimates the decrease of $\rho$ by the worst case $\alpha\ge\alpha_{\min}=1/f_{\max}$, that is, by the maximum curvature $f_{\max}=a/b^{2}$ of the evader's path. The bound therefore grows as $b$ decreases, although the numerical value is smaller for the ellipse than for the circle.

\subsection{Numerical location of the initial position where captured for $n<1$}\label{sec:num-capture}
Proposition~\ref{prop:capture} makes two predictions for small initial separations $\rho_0$: the initial angle $\zeta_0^{*}$ of a captured solution satisfies $(\pi-\zeta_0^{*})/(f(\varphi_0)\rho_0)\to1/(2+n)$ by \eqref{eq:separatrix-asym}, and the corresponding pursuer lies just outside the evader's ellipse, with $q(\mathbf{P})-1$ given to leading order by \eqref{eq:overshoot}. We test both for $a=1$, $b=0.5$, $n=0.5$ and $\varphi_0=\pi/2$, so that the evader starts at $(a,0)$, $f(\varphi_0)=a/b^{2}=4$ and $1/(2+n)=0.4$.

For each fixed $\rho_0$ the captured angle is found by bisection in $\zeta_0$. By the proof of Proposition~\ref{prop:capture}, a solution starting with $\zeta_0$ slightly below $\pi$ leaves the corner in one of two ways: $\zeta$ decreases away from $\pi$ (case (a) in the proof of Proposition~\ref{prop:capture}), or $\zeta$ crosses $\pi$ (case (b)). The captured angle $\zeta_0^{*}$ is the boundary between the two sets of initial angles, so bisection on the way in which $\zeta$ leaves the corner converges to it; along each computed $\zeta_0^{*}$ the numerical solution stays in the corner until $\rho<10^{-13}$. We then repeat this for $\rho_0=0.05,\,0.02,\,0.01,\,0.005$. Table~\ref{tab:separatrix} lists, for each $\rho_0$, the captured angle $\zeta_0^{*}$, the normalized angle $(\pi-\zeta_0^{*})/(f(\varphi_0)\rho_0)$, the value $q(\mathbf{P})$ at the corresponding initial position $\mathbf{P}=(a-\rho_0\sin\zeta_0^{*},\,-\rho_0\cos\zeta_0^{*})$ of the pursuer, and the ratio of $q(\mathbf{P})-1$ to its leading-order value in \eqref{eq:overshoot}. The angle $\zeta_0^{*}$ is given to seven decimal places: recomputation with a different integrator changes it by less than $2\times10^{-8}$, and the other columns are rounded accordingly.

\begin{table}[h]
\caption{Captured initial angle $\zeta_0^{*}$ obtained by bisection for
decreasing $\rho_0$, with $a=1$, $b=0.5$, $n=0.5$ and $\varphi_0=\pi/2$. The
third column should tend to $1/(2+n)=0.4$ by \eqref{eq:separatrix-asym}, and the
last column to $1$ by \eqref{eq:overshoot}}\label{tab:separatrix}
\begin{tabular*}{\textwidth}{@{\extracolsep\fill}lcccc}
\toprule%
& \multicolumn{2}{@{}c@{}}{Captured angle} & \multicolumn{2}{@{}c@{}}{Pursuer position} \\\cmidrule{2-3}\cmidrule{4-5}%
$\rho_0$ & $\zeta_0^{*}$ & $\dfrac{\pi-\zeta_0^{*}}{f(\varphi_0)\rho_0}$ & $q(\mathbf{P})$ & $\dfrac{q(\mathbf{P})-1}{\frac{n}{2+n}\,q(e^{i\theta_0})\,\rho_0^{2}}$ \\
\midrule
$0.05$  & $3.0618559$ & $0.39868$ & $1.001987$ & $0.9983$\\
$0.02$  & $3.1096096$ & $0.39979$ & $1.000320$ & $0.9997$\\
$0.01$  & $3.1255948$ & $0.39995$ & $1.000080$ & $0.9999$\\
$0.005$ & $3.1335929$ & $0.39999$ & $1.000020$ & $1.0000$\\
\bottomrule
\end{tabular*}
\end{table}

As $\rho_0$ decreases, the third column approaches $1/(2+n)=0.4$, confirming \eqref{eq:separatrix-asym}, and the last column approaches $1$, confirming \eqref{eq:overshoot}. The fourth column exceeds $1$ in every case: the captured pursuer starts outside the evader's ellipse, as asserted by Proposition~\ref{prop:capture}. The same behavior is observed for $(a,b,n)=(1,1,0.5)$, $(2,1,0.8)$ and $(1,0.5,0.9)$: at $\rho_0=0.005$, $(2+n)(\pi-\zeta_0^{*})/(f(\varphi_0)\rho_0)$ and the ratio in the last column differ from $1$ by less than $3\times10^{-5}$. In the circular case $a=b=1$, $n=0.5$ the prediction is accurate even at $\rho_0=0.1$, where $\zeta_0^{*}=3.1015849=\pi-0.0400078$ against the leading-order value $\pi-\rho_0/(2+n)=\pi-0.04$.

For $\zeta_0$ on either side of $\zeta_0^{*}$ the solution recovers and converges to the periodic solution. In every case the captured configuration has the pursuer outside the ellipse and slightly ahead of the evader, which is why the capture configuration is invisible to simulations started with the pursuer inside the evader's curve.

\subsection{Periodic solution}
\subsubsection{Convergence to a $\pi$-periodic solution for $n=0.5$.}
We next consider the subcritical case $0<n<1$. For circular pursuit with $n=0.5$ (\cite[Fig.~1]{JSIAML}), the solution trajectory terminates at the equilibrium point
\begin{align}
(\rho^*,\zeta^*)=\bigl(\sqrt{1-n^2},\,\cos^{-1}n\bigr)=\bigl(\sqrt{0.75},\,\pi/3\bigr)\approx(0.866,\,1.047).
\end{align}
Because this point is constant in $\varphi$, it is a (trivial) $\pi$-periodic solution of \eqref{equation-rev}, in agreement with Theorem~\ref{thm:stability}.

For elliptical pursuit with $n=0.5$ (\cite[Fig.~4]{JSIAML}), the solution trajectory converges to a non-trivial closed curve. Numerically, the $\rho$-coordinate along this curve is confined to an interval of the form roughly $0.4\lesssim\rho\lesssim 0.8$ in this example. This data shows $\rho$ is bounded away from $0$ and from $\infty$, consistent with the uniform upper bound $\rho\le2a$ of Theorem~\ref{thm:stability}(i) and with the range $r^{*}\le\rho\le2a$ of the periodic solution itself. In this simulation $|\mathbf{P}(\varphi_{0})|=0$ and $a=1.0$, so $2a=2.0$ and indeed $|z|\le2a$.

\subsubsection{Numerical verification of uniform convergence of $z(\varphi)$.}\label{sec:num-uniform}
Theorem~\ref{thm:stability} asserts not only the existence of a $\pi$-periodic solution of \eqref{equation-rev}, but also that every solution whose pursuer starts in the elliptical disc converges to it, uniformly with respect to the initial position, and that every persistently evading solution converges to it. To verify this stronger claim directly, we integrated the dynamical system \eqref{equation-rev} for $a=1$, $b=0.5$, $n=0.5$ on the interval $\varphi\in[\pi/2,\,\pi/2+20\pi]$ from four different initial conditions at $\varphi_0=\pi/2$:
\begin{align*}
&\text{Pursuer 1: }(\mu(\varphi_0),\zeta(\varphi_0))=(\log a,\,\pi/2),\\
&\text{Pursuer 2: }(\mu(\varphi_0),\zeta(\varphi_0))=(1.0,\,0.1),\\
&\text{Pursuer 3: }(\mu(\varphi_0),\zeta(\varphi_0))=(-0.5,\,\pi/2),\\
&\text{Pursuer 4: }(\mu(\varphi_0),\zeta(\varphi_0))=(-0.7,\,-\pi/2).
\end{align*}
Figure~\ref{Fig-mu-convergence} shows $\mu(\varphi)$ for all four initial conditions plotted on the same axes. After a short transient of at most a few periods, the four curves become visually indistinguishable and settle into a common $\pi$-periodic oscillation, confirming that $\mu(\varphi)$ converges uniformly to a single limit independent of its initial value.
\begin{figure}[h]
\centering
\includegraphics[width=0.70\columnwidth]{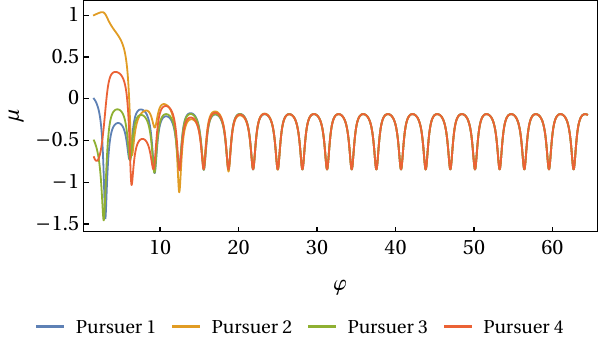}
\caption{Time evolution of $\mu(\varphi)$ for the four initial conditions of Pursuers 1--4, under the elliptical dynamical system \eqref{equation-rev} with $a=1$, $b=0.5$, $n=0.5$. After a short transient the four curves become indistinguishable, confirming uniform convergence of $\mu$ to a common $\pi$-periodic limit}
\label{Fig-mu-convergence}
\end{figure}

Figure~\ref{Fig-zeta-convergence} shows the corresponding plots of $\zeta(\varphi)$. In this case the four trajectories converge to the same limit only up to an additive integer multiple of $2\pi$: in particular Pursuer~2, whose initial $\zeta$ is close to $0$, winds once more than the others and settles near a branch offset by $+2\pi$. This extra winding occurs while the pursuer is still outside the evader's ellipse: once inside, \eqref{eq:etaform} and Corollary~\ref{cor:Kgeom}(a) give $\sin\zeta>0$ for all later $\varphi$, so $\zeta$ cannot wind any further (numerically, Pursuer~2 enters the ellipse at $\varphi-\varphi_0\approx7.36$, when $\zeta\approx2\pi+1.35$). Plotting $\zeta_2(\varphi)-2\pi$ on the same axes (dashed curve in Fig.~\ref{Fig-zeta-convergence}) makes it coincide with the limiting curve shared by Pursuers 1, 3, and 4. This $2\pi$-indeterminacy is harmless, because $\sin\zeta$ and $\cos\zeta$ are invariant under $\zeta\mapsto\zeta+2k\pi$; hence the complex variable
\begin{align}
z(\varphi)=e^{\mu(\varphi)+i\zeta(\varphi)}=e^{\mu(\varphi)}\bigl(\cos\zeta(\varphi)+i\sin\zeta(\varphi)\bigr)
\end{align}
does converge uniformly in $\varphi$ to one and the same $\pi$-periodic orbit for every initial condition.

\begin{figure}[h]
\centering
\includegraphics[width=0.63\columnwidth]{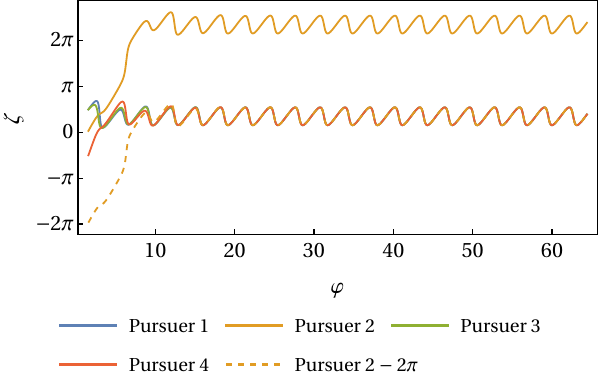}
\caption{Time evolution of $\zeta(\varphi)$ for the same four initial conditions as in Fig.~\ref{Fig-mu-convergence}. Pursuers 1, 3, 4 converge to a common limit, while Pursuer 2 converges to the same limit shifted by $+2\pi$. Subtracting $2\pi$ from $\zeta_2$ (dashed curve) makes all four trajectories coincide after the transient, so $\sin\zeta$ and $\cos\zeta$ (and hence $z=e^{\mu+i\zeta}$) converge uniformly}
\label{Fig-zeta-convergence}
\end{figure}

Translated back to physical $(x,y)$-coordinates, this uniform convergence of $z$ implies that all four pursuer trajectories are attracted onto the same closed curve inside the evader's ellipse. Figure~\ref{Fig-trajectory-convergence} shows the last single revolution of the simulation ($\varphi\in[18\pi+\pi/2,\,20\pi+\pi/2]$): the four pursuer trajectories have merged into a single closed orbit, independent of the initial separation $\mu(\varphi_0)$ and initial angle $\zeta(\varphi_0)$. This numerical experiment therefore provides strong direct evidence for the uniqueness and attractivity of the $\pi$-periodic solution established in Theorem~\ref{thm:stability}; Pursuers 1 and 3 start inside the evader's ellipse, the situation of Theorem~\ref{thm:stability}(ii), whereas Pursuers 2 and 4 start outside it, so that their convergence illustrates Theorem~\ref{thm:stability}(iii).

\begin{figure}[h]
\centering
\includegraphics[width=0.85\columnwidth]{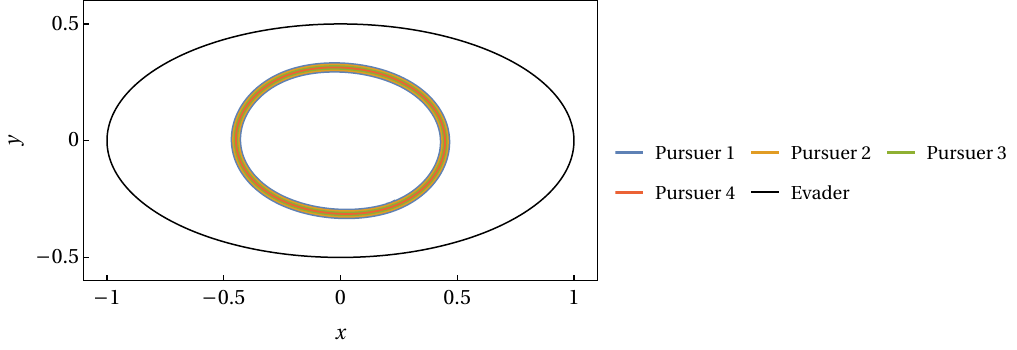}
\caption{Pursuer trajectories reconstructed from $(\mu,\zeta)$ for the four initial conditions, plotted over the final revolution $\varphi\in[18\pi+\pi/2,\,20\pi+\pi/2]$ of the simulation. All four pursuer trajectories have collapsed onto the same closed curve inside the evader's ellipse, providing direct physical evidence of convergence to the unique $\pi$-periodic solution. The four pursuer curves are drawn with different line widths only so that they remain distinguishable where they coincide}
\label{Fig-trajectory-convergence}
\end{figure}

Proposition~\ref{prop:floquet} can be checked on the periodic solution itself. Iterating the half-period map for $a=1$, $b=0.5$, $n=0.5$ from $\varphi_{0}=\pi/2$ converges to $(\rho^{*},\zeta^{*})(\varphi_{0})=(0.826491,\,1.222332)$, along which $\rho^{*}$ ranges over $[0.4294,\,0.8313]$. Quadrature along this orbit gives $I=4.125894$, while integrating the variational equation \eqref{eq:varioeq} along it gives a monodromy matrix with $\det M=0.1270789$, against $e^{-nI}=0.1270789$ from \eqref{eq:floquet}; the multipliers are $m_{1}=-0.32978$ and $m_{2}=-0.38534$, both of modulus less than one. In the circular case $a=b=1$, $n=0.5$ the same computation returns $\zeta^{*}=\arccos n$ and $I=3.627599=\pi/\sqrt{1-n^{2}}$, with a non-real conjugate pair of modulus $e^{-nI/2}=0.403774$ as Remark~\ref{rem:floquet}(a) asserts.

\section{Conclusion}
This paper examined the pursuit-evasion problem from the perspective of dynamical systems. The derivation of the dynamical system is possible because the parametrization of the evader does not change the shape of the pursuer's trajectory. The ellipse differs from the circle in that its curvature $f$ oscillates with period $\pi$: the resulting dynamical system is non-autonomous and has no equilibrium (Remark~\ref{rem:noeq}), and the asymptotically stable equilibrium of the circular case is replaced by a $\pi$-periodic solution, which nevertheless remains exponentially stable. By reformulating the distance between the two players as an exponential function and analyzing its exponent — the logarithmic distance — we derived an upper bound for the capture time when $n>1$. Moreover, we derived a single complex ODE from this dynamical system and used it to show that, for $0<n<1$, the squared distance between two solutions is non-increasing and decreases strictly over every interval; combined with the compactness of the elliptical disc and with its forward invariance, which follows from an exact equation for the depth of the pursuer inside that disc, this yields at one stroke the existence, the uniqueness and the attractivity of a $\pi$-periodic solution for every solution that avoids capture. Linearizing along this periodic solution shows further that it is hyperbolic, the product of its Floquet multipliers being the exponential of minus the accumulated reciprocal time-to-intercept over half a period, so that the convergence is eventually exponential. We also showed that capture does occur for $n<1$, but never from inside the elliptical disc: captured solutions start just outside the evader's curve on a nearly head-on course, and for small initial separations we located their initial angle asymptotically through a repelling stationary value of the desingularized depth equation.

\section{Future Work}
There are three areas for improvement in the results established here. First, the upper bound for the capture time given by Theorem~\ref{thm:Finite-time} for $n>1$ is about three times the numerically observed value in the circular example of Section~\ref{sec:verification} and about fourteen times in the elliptical one. Tightening this estimate, for instance by exploiting the oscillatory structure of $f(\varphi)$, remains an open problem. Second, we have not theoretically proven that capture is impossible when $n=1$. In the same way as in the proof of Theorem~\ref{thm:Finite-time}, now using $\rho'\ge-(1+n)\alpha_{\max}$, one obtains the lower bound
\begin{equation}
    \frac{f_{\min}\,e^{\mu(\varphi_0)}}{n +1}\le \varphi_B - \varphi_0 .
\end{equation}
The left-hand side does not diverge as $n\to 1$, so a different approach is needed. Finally, by Proposition~\ref{prop:capture} no positive lower bound on $\rho$ can hold uniformly over \emph{all} initial data. For pursuers starting \emph{inside} the evader's curve, the compactness argument in the proof of Theorem~\ref{thm:stability} does yield a positive lower bound on $\rho$ for $\varphi\ge\varphi_{0}+\pi$, uniform with respect to the initial position, and every such solution satisfies $\liminf_{\varphi\to\infty}\rho=r^{*}$ by Theorem~\ref{thm:stability}(ii); neither argument, however, gives an explicit value. An explicit lower bound, together with the precise value of $r^{*}=\min_{\varphi}|z^{*}(\varphi)|$ as a function of $a$, $b$ and $n$, is left for future work.

Theorem~\ref{thm:Finite-time} holds whenever $f$ is bounded, whereas the arguments of Sections~\ref{sec:capture} and~\ref{sec:periodic} use the central symmetry of the ellipse and the quadratic form $q$ that defines it. Extending them to pursuit problems in which the evader follows more general closed curves (such as those studied in \cite{Azamov,Kuchkarov}) is an open problem, with an eye toward characterizing the long-term shape of the pursuer's trajectory, rather than merely the feasibility of capture. Such extensions would be relevant to tracking objects moving along diverse closed trajectories, particularly orbiting bodies.

\section*{Acknowledgments}
This paper was completed under the guidance of Professor Toru Ohira of the Graduate School of Mathematics, Nagoya University.

\section*{Statements and Declarations}

\begin{itemize}
\item \textbf{Funding:} This work was financially supported by the Japan Science and Technology Agency (JST) SPRING, Grant Number JPMJSP2125.
\item \textbf{Competing interests:} The author has no competing interests to declare that are relevant to the content of this article.
\item \textbf{Data availability:} No datasets were generated or analyzed during the current study. All numerical results in Section~\ref{sec:verification} can be reproduced from the equations, parameter values and initial conditions stated there.
\item \textbf{Code availability:} The Mathematica code used for the numerical computations in Section~\ref{sec:verification}, which was generated with the assistance of a large language model (see Use of AI tools below), is available from the author on reasonable request.
\item \textbf{Use of AI tools:} The author used Claude Opus 4.7 and Claude Opus 5.5 (Anthropic), large language models, (i) as a discussion partner in developing the arguments of Sections~\ref{sec:capture} and~\ref{sec:periodic}, including proof strategies and the checking of proofs; (ii) to generate the Mathematica code for the numerical computations in Section~\ref{sec:verification}; and (iii) for copy editing. All mathematical statements and proofs were verified by the author, and the generated code and its output were checked by the author against the analytical results of Section~\ref{sec:analysis}. The author takes full responsibility for the correctness of the results and for the final content of this manuscript.
\item \textbf{Author contributions:} The sole author conceived the study, developed the proofs, and wrote the manuscript.
\end{itemize}


\end{document}